\documentclass[11pt]{article}
\usepackage{amsfonts, amssymb, amsthm, amsmath, latexsym, bm}
\usepackage{enumerate, enumitem}
\usepackage{booktabs}
\usepackage{subfigure, epsfig, color, caption, graphicx}
\usepackage{float}
\usepackage{authblk, mathrsfs, indentfirst}
\usepackage{hyperref, cleveref}
\hypersetup{colorlinks=true, linkcolor=green, filecolor=magenta, urlcolor=cyan}
\usepackage[a4paper, left=2.5cm, right=2.5cm, top=2.5cm, bottom=2.5cm]{geometry}
\allowdisplaybreaks
\numberwithin{equation}{section}
\newtheorem{theorem}{Theorem}
\newtheorem{definition}{Definition}[section]
\newtheorem{lemma}{Lemma}[section]

\newtheorem{corollary}{Corollary}[section]
\newtheorem{example}{Example}[section]

\newtheorem{remark}{Remark}[section]
\newcommand{\keywords}[1]{\small\textbf{\textit{Keywords---}}#1}

\title{The Stability of the Backward Problem for Photoacoustic Imaging in Attenuating Media via Carleman Estimates}

\date{}

\author[1]{Qihang Chen}
\author[2]{Zhiyuan Li\footnote{Corresponding author: lizhiyuan@nbu.edu.cn, supported by the National Natural Science Foundation of China (no. 12271277), Ningbo Youth Leading Talent Project (no. 2024QL045).}}
\author[3]{Song Xu}

\affil[1,2,3]{School of Mathematics and Statistics, Ningbo University, Ningbo 315211, China}

\begin{document}
	
\maketitle
	
\abstract{This paper investigates the backward problem in time for photoacoustic tomography (PAT) in attenuating media. It is well-established that photoacoustic imaging in attenuating media can be accurately modeled by spatial fractional-order damping. This inverse problem is ill-posed in the sense of Hadamard. In this work, we construct a novel class of Carleman estimates independent of spatial variables, and by virtue of these estimates, we establish conditional stability estimates for this problem for the first time. Building upon this, we propose a Tikhonov-type regularization functional and derive its associated adjoint system. Furthermore, leveraging the established conditional stability results, we derive the convergence rate of the proposed regularization approach. Finally, we validate the effectiveness of our theoretical findings through extensive numerical experiments. 
}



\section{Introduction and main results}
The classical wave equation serves as the fundamental governing equation for photoacoustic tomography (PAT) in ideal non-attenuating media \cite{Ammari2010,Ren2015}. In practical PAT applications, however, acoustic attenuation is ubiquitous and cannot be neglected in high-precision image reconstruction, which necessitates the replacement of the classical wave equation with damped wave models. Damped wave equations have been widely adopted to characterize attenuating wave propagation across diverse physics and engineering disciplines \cite{hatano1998dispersive,kilbas2006theory,metzler2000random}. Traditional PAT models typically assume a velocity-proportional damping mechanism, which introduces a conventional \(\partial_t \Delta u\) damping term. Nevertheless, practical acoustic attenuation exhibits prominent frequency-dependent characteristics that generally follow a power-law distribution, as extensively verified in existing PAT and wave propagation studies \cite{cai2018fractional}. {The inverse problem studied in this paper does not employ the boundary measurement approach usually used in quantitative photoacoustic tomography, see \cite{Ren2015}. Motivated by the equation, we use terminal time observations of the displacement and velocity.}

To accurately capture such power-law frequency-dependent damping behavior, various modified wave models have been developed in the literature. Among these modeling strategies, the fractional Laplacian-based wave equation, which employs fractional powers of the Laplacian or general second-order uniformly elliptic operators, has emerged as the most effective and widely used framework \cite{ChenWen2005,10.1121/1.428630}. Against this backdrop, this paper investigates the inverse problem for PAT under frequency-dependent attenuation, covering both classical integer-order and generalized fractional-order damped wave imaging models. The core of this PAT inverse problem is to recover the unknown initial pressure distribution of the imaging medium from finite measured wave field data. Specifically, when wave field measurements are acquired over the entire computational domain \(\Omega\) at the fixed terminal time \(T\) \cite{2024arXiv240518082D}, the forward wave propagation problem is converted to a time-reversed backward inverse problem, which constitutes the key research object of this work.

Based on the above physical mechanism and modeling framework, we research a generalized spatial damped wave equation. Let \(\Omega \subset \mathbb{R}^n\) denote a bounded open domain in the \(n\)-dimensional Euclidean space, and let \(T > 0\) be a fixed terminal observation time. We define the space-time cylinder for the imaging system as \(\Omega_T := \Omega \times (0, T)\). In this domain, we study a generalized multi-term damped wave equation with mixed integer- and fractional-order damping terms, formulated as follows:
\begin{equation}\label{eq-main}
\begin{cases}
\partial_t^2 u
+{\mathcal A_\delta}u
+\partial_t{\mathcal A_\beta}u
+a(x,t)\partial_t{\mathcal A_{\gamma/2}}u=0,
& (x,t)\in\Omega_T,\\
u(x,0)=u_0(x),
& x\in\Omega,\\
\partial_tu(x,0)=0,
& x\in\Omega,\\
{\mathcal B_{\mathcal A}u=0},
& {t\in(0,T)}.
\end{cases}
\end{equation}
where the coefficient \(a \in L^\infty(\Omega_T)\) is a predefined bounded function defined on the space-time domain \(\Omega_T\). The exponents \(\delta,\beta,\gamma\) are non-negative real numbers in $[0,1]$.

{
The operator family $\mathcal A_s$ and the associated homogeneous
condition $\mathcal B_{\mathcal A}u=0$ are specified by either of
the following definitions.

\begin{definition}[Spectral Definition]
\label{def:spectral_laplacian}
Let $0<s<1$, and let $\{(\lambda_k,\phi_k)\}_{k=1}^\infty$
be the Dirichlet eigenpairs
\[
-\Delta\phi_k=\lambda_k\phi_k
\quad\text{in }\Omega,
\qquad
\phi_k=0
\quad\text{on }\partial\Omega,
\]
where $\{\phi_k\}_{k=1}^\infty$ is an orthonormal basis of
$L^2(\Omega)$. Following \cite{Grubb2016Spectral}, we define
\[
\mathcal A_s v
=
(-\Delta)^s v
=
\sum_{k=1}^\infty
\lambda_k^s(v,\phi_k)_{L^2(\Omega)}\phi_k.
\]
Here $(-\Delta)^s$ is the spectral operator used in
Section~\ref{sec2}, and
\[
D(\mathcal A_s)=D((-\Delta)^s)
=
\left\{
v\in L^2(\Omega):
\sum_{k=1}^\infty
\lambda_k^{2s}|(v,\phi_k)_{L^2(\Omega)}|^2<\infty
\right\}.
\]
The associated homogeneous Dirichlet condition is encoded in
this realization. Whenever the boundary trace is defined,
$\mathcal B_{\mathcal A}u=0$ means $u=0$ on $\partial\Omega\times(0,T).
$
\end{definition}

\begin{definition}[Riesz Definition]
\label{def:riesz_laplacian}
Let $0<s<1$, and let $\widetilde v$ denote the zero extension
of $v$ from $\Omega$ to $\mathbb R^n$.
Following \cite{Grubb2015FractionalDomains,Grubb2018FractionalOperators},
we define, for sufficiently regular functions,
\[
\mathcal A_s v(x)
=
\mathcal L_s v(x)
=
C_{n,s}\,\operatorname{P.V.}
\int_{\mathbb R^n}
\frac{\widetilde v(x)-\widetilde v(y)}
{|x-y|^{n+2s}}\,dy,
\qquad x\in\Omega.
\]
The normalization constant is
\begin{equation}\label{eq_def}
C_{n,s}
=
\frac{2^{2s}s\Gamma\left(\frac{n+2s}{2}\right)}
{\pi^{n/2}\Gamma(1-s)}.
\end{equation}
Here $\mathcal L_s$ is the zero-exterior Riesz realization
introduced in Section~\ref{sec2}, with
\[
D(\mathcal A_s)=D(\mathcal L_s)
=
\left\{
v\in\mathcal H_0^s(\Omega):
\left.(-\Delta)^s_{\mathbb R^n}\widetilde v\right|_\Omega
\in L^2(\Omega)
\right\}.
\]
The full-space operator in this domain condition is understood
distributionally, and the associated realization is defined
through the energy form specified in Section~\ref{sec2}.
The homogeneous condition $\mathcal B_{\mathcal A}u=0$ means
$u=0$ in $(\mathbb R^n\setminus\Omega)\times(0,T)$.
\end{definition}

For both realizations, we set $\mathcal A_0=I$ and
$\mathcal A_1=-\Delta_D$, where $-\Delta_D$ denotes the usual
Dirichlet Laplacian. When the spatial operator orders
$\delta$, $\beta$, and $\gamma/2$ belong to $\{0,1\}$,
\eqref{eq-main} reduces to the corresponding classical
integer order.
}

The generalized damped wave equation \eqref{eq-main} studied in this paper can be reduced to various classical PDE models by taking different values of the parameters $\gamma, \beta, \delta$, and related research has achieved abundant results. The following will systematically sort out the existing research progress according to parameter classification.

When $\delta=\beta=1$ and $a=0$, model \eqref{eq-main} reduces to the classical strongly damped wave equation. 
\begin{equation}
\partial_t^2 u -\Delta u - \partial_t \Delta u  = 0.
\end{equation}
For the forward problem with initial data, existing works have established criteria for global existence/non-existence and uniqueness/non-uniqueness, and systematically studied the basic properties of solutions such as asymptotic behavior and regularity. Among them, the global existence and asymptotic behavior of strong solutions were established by Webb \cite{webb1980existence}, the blow-up analysis results can be found in Ono and Ohta \cite{ono1997global,ohta1998remarks}, and the threshold for existence and non-existence of solutions was given by Ma et al. \cite{ma2018energy} with nonlinear source term. In the field of inverse problem research, Tuan et al. \cite{TUAN201769} discussed the backward problem for this strongly damped wave equation and first proved its ill-posedness in the sense of Hadamard.

When $\beta\in(0,1)$, equation \eqref{eq-main} corresponds to the Chen–Holm model \cite{chen2004fractional},
\begin{equation}
    \partial_t^2 u -\Delta u + \partial_t (-\Delta)^\beta u  = 0.
\end{equation}
 When $\beta\in(0,1/2)$, Carvalho et al. \cite{carvalho2002attractors} established the global solvability of the forward problem; for the homogeneous case, DaLuz et al. \cite{da_luz2015asymptotic} and Karch \cite{karch2000selfsimilar} characterized the long-time asymptotic behavior of global solutions in various norms through Fourier analysis methods. In  \cite{Kaltenbacher_2021}, Kaltenbacher et al. further studied the initial value inversion problem based on trace observations in the time and space fractional cases, established stability results, and proposed numerical solution algorithms.

Regarding the backward problem for this equation in the general case, Song et al. \cite{SONG2024177} conducted a preliminary exploration and gave a qualitative analysis of stability and instability in the general case, but did not establish a rigorous quantitative conditional stability theory. Even in the classical case of the strongly damped wave equation, existing studies \cite{TUAN201769} have failed to provide universal stability estimates. Therefore, there remains a significant theoretical gap in the study of conditional stability for the backward problem of generalized fractional damped wave equations.

As a typical ill-posed problem, the backward problem has been fully studied in parabolic systems \cite{isakov2017inverse}. For the backward problem of time-fractional wave equations, the conditional stability can be found in the work of Yamamoto et al. \cite{Cen2025, Chorfi12082024, Floridia_2020,doi:10.1137/22M1529105,wen2023solving}, and existing methods are mostly based on eigenvalue expansion techniques. Yamamoto \cite{Yamamoto_2009} first introduced the Carleman estimate method into the study of backward problems for parabolic equations, providing a powerful tool for the stability analysis of ill-posed problems. Subsequently, Zhang et al. \cite{zhang2025conditional} and Jia et al. \cite{JIA20181} applied this method to backward problems in dynamic cardiac electrophysiology modeling and space-fractional diffusion equations, respectively, but these equations are essentially variants of parabolic equations. Inspired by this, this paper extends the Carleman estimate method to generalized space-fractional damped wave equations and establishes corresponding theoretical results by constructing adapted exponential weight functions and integration by parts. 

Before giving the main results, we recall that \( H^k(0,T; X) \) is the Bochner--Sobolev space of \( X \)-valued functions with \( k \)-th order time weak derivatives in \( L^2(0,T; X) \); \( H^s(\Omega) \) is the Sobolev space of order \( s\ge0 \); \( L^\infty(\Omega_T) \) denotes essentially bounded functions on \( \Omega_T = \Omega\times(0,T) \). {
For a function $v$ defined on $\Omega$, let $\widetilde v$
denote its zero extension to $\mathbb{R}^n$, with
$\widetilde v=v$ in $\Omega$ and $\widetilde v=0$ in
$\mathbb{R}^n\setminus\Omega$.
} The more details for these function spaces are postponed in Section \ref{sec2}. We are now ready to state the first main theorem.
\begin{theorem}\label{thm3}
Let $u \in H^2(0,T;L^2(\Omega))
\cap H^1(0,T;D((-\Delta)^\beta))$ be a solution to \eqref{eq-main} with Definition \ref{def:spectral_laplacian} and satisfy the following a priori regularity condition:
\[
\|u\|_{H^{1}(0,T;D((-\Delta)^\beta))}\leq M.
\]
If $\delta \leq \beta$, $\gamma \leq \beta$ and $a\in L^{\infty}(\Omega_T)$, then for any $t_0 \in (0,T)$, there exist constants $\theta_1 \in (0,1)$ and $C>0$, depending on {$T$, $\Omega$, $t_0$, $t_1$, $t_2$, $\beta$, $\delta$, $\gamma$, $\|a\|_{L^\infty(\Omega_T)}$, $\lambda$, and $\theta_1$}, such that
\[
\|u(\cdot, t_0)\|_{L^2(\Omega)}+\|\partial_t u(\cdot,t_0)\|_{L^2(\Omega)} \leqslant C M^{1-\theta_1} \left(\|u(\cdot, T)\|_{D((-\Delta)^\beta))}+\|\partial_t u(\cdot, T)\|_{H^\beta(\Omega)}\right)^{\theta_1},
\]
where $\theta_1=\frac{e^{\lambda t_0}-e^{\lambda t_2}}{e^{\lambda T}-e^{\lambda t_2}+\kappa}$ with $\kappa >0$, and $t_2$ is a parameter satisfying {$0 <t_1< t_2 < t_0$.}
\end{theorem}
{
\begin{theorem}
\label{thm:riesz_stability}
Let $0<\beta<1$, $0\leq\delta\leq\beta$, $0\leq\gamma\leq\beta$, and $a\in L^\infty(\Omega_T)$. Assume that
\[
\delta=\beta\quad\text{or}\quad 2\delta<\beta+\frac12.
\]
Let $\mathcal L_s$ and $\mathcal H_0^s(\Omega)$ be as defined in Section~\ref{sec2}. Suppose that $u\in H^2(0,T;L^2(\Omega))\cap H^1(0,T;D(\mathcal L_\beta))$
solves \eqref{eq-main} with Definition \ref{def:riesz_laplacian} and satisfies
\[
\|u\|_{H^1(0,T;D(\mathcal L_\beta))}\leq M,
\qquad M>0.
\]
Fix $t_0\in(0,T)$ and $0<t_1<t_2<t_0$. For a fixed sufficiently large $\lambda>0$, set $\varphi(t)=e^{\lambda t}$. Then, for every $\theta_1=\frac{e^{\lambda t_0}-e^{\lambda t_2}}{e^{\lambda T}-e^{\lambda t_2}+\kappa}$ with $\kappa>0$, there exists $C>0$, independent of $u$ and $M$, such that
\[
\begin{aligned}
\|u(\cdot,t_0)\|_{L^2(\Omega)}+\|\partial_tu(\cdot,t_0)\|_{L^2(\Omega)}
\leq C M^{1-\theta_1}
\left(\|u(\cdot,T)\|_{D(\mathcal L_\beta)}
+\|\partial_tu(\cdot,T)\|_{\mathcal H_0^\beta(\Omega)}\right)^{\theta_1}.
\end{aligned}
\]
The constant may depend on $T$, $\Omega$, $t_0$, $t_1$, $t_2$, $\beta$, $\delta$, $\gamma$, $\|a\|_{L^\infty(\Omega_T)}$, $\lambda$, and $\theta_1$.
\end{theorem}
}
In terms of numerical algorithms, \cite{hou2026modified,LIU201884} conducted numerical experiments of Tikhonov regularization for the backward problem of the heat equation. Based on the established conditional stability results, this paper constructs the corresponding Tikhonov regularization functional, analyzes the convergence of the algorithm, and gives the convergence rate for this problem.

{
Let the regularized state be defined by
\[
v_0^{\alpha,\sigma}
=
\operatorname*{arg\,min}_{v\in\mathcal K_{M,\epsilon}(t_0)}
\left\{
\|\mathcal F_{t_0}v-y^\sigma\|_{X_0}^2
+\alpha\|v\|_{X_0}^2
\right\},
\]
where $\mathcal K_{M,\epsilon}(t_0)$ is defined in
\eqref{eq:admissible_states}.
}
{
Let $\mathcal F_{t_0}$ map the state at $t_0$ to the terminal
state at $T$}.
\begin{theorem}
\label{thm:main}
Assume the conditional stability estimate holds
{for Definition \ref{def:spectral_laplacian}}
and that the true solution satisfies
$\|u^\dagger\|_{H^1(0,T;{D((-\Delta)^{\beta+\epsilon})})}
\le M$,
{where $M>0$ and $\epsilon>0$ are fixed}.
Let $v_0 =
(u{(\cdot,t_0)},
 \partial_t u{(\cdot,t_0)})$
denote the state taking values in
${X_0=}
L^2(\Omega)\times L^2(\Omega)$,
and $v_0^\dagger$ stand for the exact historical state
corresponding to the true solution
{at a fixed time} $t_0\in(0,T)$.
Here $\sigma>0$ denotes the noise level of the observed data.
{
And assume that the noisy observation
$y^\sigma\in X_0$ satisfies
\[
\|y^\sigma-\mathcal F_{t_0}v_0^\dagger\|_{X_0}\leq\sigma.
\]}
In addition, {assume $0<\beta\leq1$,
$0\leq\delta\leq\beta$, and $a=0$}.
Choose the regularization parameter as $\alpha=c\sigma^2$
with an arbitrary constant $c>0$.
{Fix an admissible exponent $\theta_1$
from Theorem~\ref{thm3}, as specified in
Remark~\ref{remark4.1}.}
Then there exists a constant $C>0$,
{depending on $M$, $c$, $\epsilon$
and the fixed parameters of the conditional stability
estimate, and independent of $\sigma$},
such that
\[
\|v_0^{\alpha,\sigma}-v_0^\dagger
\|_{L^2(\Omega)\times L^2(\Omega)}
\leq C\,\sigma^\theta,
\]
where
$\theta=\theta_1\frac{\epsilon}{\beta+\epsilon}\in(0,1)$.
\end{theorem}

In summary, the main contributions of this paper are as follows:
\begin{enumerate}
\item The Carleman estimate for generalized fractional damped wave equations is established for the first time.
\item Based on the above Carleman estimate, the conditional stability of the backward problem for this equation is rigorously proved under a priori regularity assumptions.
\item According to the established conditional stability, the Tikhonov regularization functional is constructed, and the convergence of the algorithm is analyzed using conditional stability, giving the convergence rate for this problem for the first time.
\end{enumerate}

The remainder of this paper is structured as follows. In Section~\ref{sec2}, we recall the necessary function spaces and preliminary lemmas, including a weighted integral inequality that plays a crucial role in the Carleman estimate. Section~\ref{sec3} is devoted to the derivation of two Carleman estimates for the strongly damped wave operator with space‑fractional damping in two definition. Based on these estimates, Section~\ref{sec4} rigorously proves a conditional stability estimate for the backward problem. In Section~\ref{sec5}, we formulate the Tikhonov regularization for the ill‑posed backward problem, derive the gradient via the adjoint method, and analyze the convergence rate under a suitable parameter choice rule. Section~\ref{sec6} presents numerical experiments that validate the theoretical findings and investigate the influence of the fractional order, the initial time $t_0$, and the noise level on the reconstruction accuracy. Finally, Section~\ref{sec7} concludes the paper.

\section{Preliminaries}\label{sec2}
Throughout this paper, we adopt the notations introduced above and assume
that {$\Omega\subset\mathbb R^n$}.
We first specify the standard function spaces used throughout our analysis.

The standard $L^2$ spaces on $\Omega$ and $\Omega_T$ are denoted by
$L^2(\Omega)$ and $L^2(\Omega_T)$, respectively.

{
For $r\geq0$, let $H^r(\Omega)$ denote the standard Sobolev space
obtained by restricting $H^r(\mathbb R^n)$ to $\Omega$.
For $0<s<1$, an equivalent norm is
\[
\|v\|_{H^s(\Omega)}^2
=
\|v\|_{L^2(\Omega)}^2
+
\int_\Omega\int_\Omega
\frac{|v(x)-v(y)|^2}{|x-y|^{n+2s}}\,dx\,dy.
\]
}

For a Banach space $X$, $C([0,T];X)$ denotes the space of continuous
functions from $[0,T]$ into $X$, $L^2(0,T;X)$ the space of square
Bochner integrable functions, and $H^1(0,T;X)$ the space of functions
whose weak derivative also belongs to $L^2(0,T;X)$.

{
\medskip
\noindent\textbf{Spectral energy space.}
Let $(-\Delta)^s$ be the spectral Dirichlet operator in
Definition~\ref{def:spectral_laplacian}.
For $0<s\leq1$, we use the notation
\[
H_0^s(\Omega):=D((-\Delta)^{s/2}),
\]
equipped with the norm
\[
\|v\|_{H_0^s(\Omega)}^2
=
\|v\|_{L^2(\Omega)}^2
+
\|(-\Delta)^{s/2}v\|_{L^2(\Omega)}^2.
\]
In particular, $H_0^1(\Omega)$ is the usual Sobolev space with
zero boundary trace.
For $v\in D((-\Delta)^s)$ and $w\in H_0^s(\Omega)$, the spectral
construction gives
\[
\bigl((-\Delta)^s v,w\bigr)_{L^2(\Omega)}
=
\bigl((-\Delta)^{s/2}v,
      (-\Delta)^{s/2}w\bigr)_{L^2(\Omega)}.
\]

\medskip
\noindent\textbf{Riesz energy space.}
Let $\mathcal L_s$ be the zero-exterior Riesz operator in
Definition~\ref{def:riesz_laplacian}, and let $\widetilde v$
denote the zero extension used there.
For $0<s<1$, define
\[
\mathcal H_0^s(\Omega)
=
\left\{
v\in L^2(\Omega):
\widetilde v\in H^s(\mathbb R^n)
\right\}.
\]
For real-valued functions, the associated bilinear form is
\[
\mathcal E_s(v,w)
=
\frac{C_{n,s}}2
\int_{\mathbb R^n}\int_{\mathbb R^n}
\frac{
(\widetilde v(x)-\widetilde v(y))
(\widetilde w(x)-\widetilde w(y))
}{
|x-y|^{n+2s}
}\,dx\,dy,
\]
where $C_{n,s}$ is defined in \eqref{eq_def}.
We equip $\mathcal H_0^s(\Omega)$ with the energy norm
\[
\|v\|_{\mathcal H_0^s(\Omega)}^2
=
\|v\|_{L^2(\Omega)}^2+\mathcal E_s(v,v),
\]
which is equivalent to
$\|\widetilde v\|_{H^s(\mathbb R^n)}^2$.

The operator $\mathcal L_s$ is positive and self-adjoint in
$L^2(\Omega)$, with associated form $\mathcal E_s$.
For $v\in D(\mathcal L_s)$ and $w\in\mathcal H_0^s(\Omega)$,
\[
(\mathcal L_s v,w)_{L^2(\Omega)}
=
\mathcal E_s(v,w).
\]
Moreover,
\[
D(\mathcal L_s^{1/2})=\mathcal H_0^s(\Omega),
\]
and, for every $v\in\mathcal H_0^s(\Omega)$,
\[
\|\mathcal L_s^{1/2}v\|_{L^2(\Omega)}^2
=
\mathcal E_s(v,v)
=
\|(-\Delta)_{\mathbb R^n}^{s/2}\widetilde v
\|_{L^2(\mathbb R^n)}^2.
\]

\medskip
For either realization $\mathcal A_s$ in
Definitions~\ref{def:spectral_laplacian}
and~\ref{def:riesz_laplacian}, the operator domain specified
there is equipped with the graph norm
\[
\|v\|_{D(\mathcal A_s)}^2
=
\|v\|_{L^2(\Omega)}^2
+
\|\mathcal A_s v\|_{L^2(\Omega)}^2.
\]
Thus this convention gives the norms on
$D((-\Delta)^s)$ and $D(\mathcal L_s)$ used below.

For $0<s<1$, the spaces $H_0^s(\Omega)$ and
$\mathcal H_0^s(\Omega)$ coincide as sets and have equivalent norms.
At $s=\frac12$, their common space is the Lions--Magenes space
\[
H_{00}^{1/2}(\Omega)
=
\left\{
v\in H^{1/2}(\Omega):
\int_\Omega
\frac{|v(x)|^2}{\operatorname{dist}(x,\partial\Omega)}\,dx
<\infty
\right\}.
\]
}

Beyond the function spaces and operator identities above, we state
a key integral lemma that will be used in the derivation of the
Carleman estimate and the proof of conditional stability for our
inverse problem.
\begin{lemma}
\label{lem:terminal}
Let $\lambda>0$, $s>0$, and define the weight function $\varphi(t)=e^{\lambda t}$ on $[0,T]$.
For any $w\in L^2(\Omega_T)$, set
\[
v(x,t)=\int_0^t w(x,\tau)\,d\tau.
\]
Then the following estimate holds:
\[
\begin{aligned}
&\int_{\Omega_T} \varphi(t) e^{2s\varphi(t)}\Bigl|\int_0^t w(x,\tau)\,d\tau\Bigr|^{2}\,dxdt
\\
\le& \frac{e^{2s\varphi(T)}}{s\lambda}\int_\Omega\Bigl|\int_0^T w(x,\tau)\,d\tau\Bigr|^{2}\,dx
   + \frac{1}{s^{2}\lambda^{2}}\int_{\Omega_T} \frac{e^{2s\varphi(t)}}{\varphi(t)}|w(x,t)|^{2}\,dxdt.
\end{aligned}
\]
\end{lemma}

\begin{proof}
From $\varphi'=\lambda\varphi$, we have
\begin{equation}
\frac{d}{dt}\bigl[e^{2s\varphi(t)}\bigr] = 2s\lambda\,\varphi(t)e^{2s\varphi(t)}.
\label{eq:deriv_weight}
\end{equation}
Since $v_t=w$, the product rule gives
\[
\partial_t\bigl(e^{2s\varphi} v^2\bigr) = 2s\lambda\varphi e^{2s\varphi} v^2 + 2e^{2s\varphi} v w.
\]
Define $k(t)$ by $\frac{d}{dt}\bigl[k(t)e^{2s\varphi(t)}\bigr] = \varphi(t)e^{2s\varphi(t)}$.
Using \eqref{eq:deriv_weight} we see that the right‑hand side equals $\frac{1}{2s\lambda}\frac{d}{dt}[e^{2s\varphi}]$.
Integrating from $0$ to $t$ with $k(0)e^{2s\varphi(0)}=0$ yields
\begin{equation}
k(t)e^{2s\varphi(t)} = \int_0^t \varphi(\tau)e^{2s\varphi(\tau)}\,d\tau
= \frac{1}{2s\lambda}\bigl(e^{2s\varphi(t)} - e^{2s\varphi(0)}\bigr).
\label{eq:theta_explicit}
\end{equation}
Now differentiating $k e^{2s\varphi}v^2$ and using the definition of the function $k(t)$, we see that
\[
\partial_t\bigl(k e^{2s\varphi} v^2\bigr)
= \varphi e^{2s\varphi} v^2 + 2k e^{2s\varphi} v w.
\]
Integrate over $\Omega_T$ and apply the fundamental theorem of calculus, and noting $v(x,0)=0$, we arrive at the identity
\[
\int_{\Omega_T} \partial_t\bigl(k e^{2s\varphi} v^2\bigr)\,dxdt
= \int_\Omega k(T)e^{2s\varphi(T)}v(x,T)^2\,dx.
\]
Using \eqref{eq:theta_explicit} at $t=T$,
\[
k(T)e^{2s\varphi(T)} = \frac{1}{2s\lambda}\bigl(e^{2s\varphi(T)} - e^{2s\varphi(0)}\bigr).
\]
Hence
\[
\int_\Omega k(T)e^{2s\varphi(T)}v(\cdot,T)^2\,dx
= \frac{1}{2s\lambda}\bigl(e^{2s\varphi(T)}-e^{2s\varphi(0)}\bigr)\int_\Omega v(\cdot,T)^2dx,
\]
which combined with the last two equalities, we obtain
\begin{equation}
\frac{1}{2s\lambda}\bigl(e^{2s\varphi(T)}-e^{2s\varphi(0)}\bigr)\int_\Omega v(\cdot,T)^2\,dx
= \int_{\Omega_T} \varphi e^{2s\varphi} v^2\,dxdt + 2\int_{\Omega_T} k e^{2s\varphi} v w \,dxdt.
\label{eq:main_identity}
\end{equation}
Substituting \eqref{eq:theta_explicit} into the last term of \eqref{eq:main_identity}, we obtain the following:
\[
2\int_{\Omega_T} k e^{2s\varphi} v w\,dxdt
= \frac{1}{s\lambda}\int_{\Omega_T} e^{2s\varphi} v w\,dxdt
  - \frac{e^{2s\varphi(0)}}{s\lambda}\int_{\Omega_T} v w\,dxdt.
\]
Since $w=v_t$ and $v(x,0)=0$,
\[
\int_{\Omega_T} v w\,dxdt = \int_\Omega\int_0^T \tfrac12\partial_t(v^2)\,dt\,dx
= \frac12\int_\Omega v(\cdot,T)^2\,dx,
\]
which further implies
\[
2\int_{\Omega_T} k e^{2s\varphi} v w\,dxdt
= \frac{1}{s\lambda}\int_{\Omega_T} e^{2s\varphi} v w\,dxdt
  - \frac{e^{2s\varphi(0)}}{2s\lambda}\int_\Omega v(\cdot,T)^2\,dx.
\]
Plugging this into \eqref{eq:main_identity} cancels the $e^{2s\varphi(0)}$ terms and yields
\begin{equation}
\int_{\Omega_T} \varphi e^{2s\varphi} v^2\,dxdt
= \frac{e^{2s\varphi(T)}}{2s\lambda}\int_\Omega v(\cdot,T)^2\,dx
  - \frac{1}{s\lambda}\int_{\Omega_T} e^{2s\varphi} v w\,dxdt.
\label{eq:before_I}
\end{equation}
Set $I := \int_{\Omega_T} \varphi(t) e^{2s\varphi(t)} v(x,t)^2\,dxdt$, then \eqref{eq:before_I} becomes
\begin{equation}
I = \frac{e^{2s\varphi(T)}}{2s\lambda}\int_\Omega v(\cdot,T)^2\,dx
    - \frac{1}{s\lambda}\int_{\Omega_T} e^{2s\varphi} v w\,dxdt.
\label{eq:I_eq}
\end{equation}
Now we bound the second term.  Using $-vw \le |vw|$ and the weighted Young inequality
\begin{equation}\label{youngeq}
    |vw| \le \frac{\alpha}{2}\frac{w^2}{\varphi} + \frac{1}{2\alpha}\varphi v^2, \quad \alpha>0,
\end{equation}
we obtain
\[
\begin{aligned}
- \frac{1}{s\lambda}\int_{\Omega_T} e^{2s\varphi} v w\,dxdt
\le& \frac{1}{s\lambda}\int_{\Omega_T} e^{2s\varphi} |v w|\,dxdt\\
\le& \frac{1}{s\lambda}\left( \frac{\alpha}{2}\int_{\Omega_T} \frac{e^{2s\varphi}}{\varphi} w^2\,dxdt
   + \frac{1}{2\alpha}\int_{\Omega_T} \varphi e^{2s\varphi} v^2\,dxdt \right).
\end{aligned}
\]
Substituting this estimate into \eqref{eq:I_eq}, we obtain:
\[
I \le \frac{e^{2s\varphi(T)}}{2s\lambda}\int_\Omega v(\cdot,T)^2\,dx
    + \frac{\alpha}{2s\lambda}\int_{\Omega_T} \frac{e^{2s\varphi}}{\varphi} w^2\,dxdt
    + \frac{1}{2\alpha s\lambda} I.
\]
Choose $\alpha = \frac{1}{s\lambda}$.  Then $\frac{1}{2\alpha s\lambda} = \frac12$, and the inequality becomes
\[
I \le \frac{e^{2s\varphi(T)}}{2s\lambda}\int_\Omega v(\cdot,T)^2\,dx
    + \frac{1}{2s^2\lambda^2}\int_{\Omega_T} \frac{e^{2s\varphi}}{\varphi} w^2\,dxdt
    + \frac12 I.
\]
Subtract $\frac12 I$ and multiply by $2$:
\[
I \le \frac{e^{2s\varphi(T)}}{s\lambda}\int_\Omega v(\cdot,T)^2\,dx
    + \frac{1}{s^2\lambda^2}\int_{\Omega_T} \frac{e^{2s\varphi}}{\varphi} w^2\,dxdt.
\]
Recalling the definition of $I$ and $v(x,t)=\int_0^t w(x,\tau)\,d\tau$ gives the desired estimate. We complete the proof of the lemma.
\end{proof}
\begin{remark}\label{remark1}
    Since $\varphi(t) \geq 1$, if we use the Young inequality \eqref{youngeq} without the weight function $\varphi$, that is,
    \begin{equation}
    |vw| \le \frac{\alpha}{2}w^2 + \frac{1}{2\alpha} v^2 \qquad (\alpha>0),
\end{equation}
then the estimate becomes to 
\begin{equation}\label{remarkeq}
\begin{aligned}
&\int_{\Omega_T}e^{2s\varphi(t)}\Bigl|\int_0^t w(x,\tau)\,d\tau\Bigr|^{2}\,dxdt\\
\le& \frac{e^{2s\varphi(T)}}{s\lambda}\int_\Omega\Bigl|\int_0^T w(x,\tau)\,d\tau\Bigr|^{2}\,dx
   + \frac{1}{s^{2}\lambda^{2}}\int_{\Omega_T} e^{2s\varphi(t)}|w(x,t)|^{2}\,dxdt,    
\end{aligned}
\end{equation}
which is important to the following proof.

\end{remark}

\section{Carleman Estimates}\label{sec3}
In this section, we turn to considering a Carleman estimate with a weight function that depends only on time:
\[
\varphi(t)=e^{\lambda t},\qquad \lambda>0.
\] 
{Throughout this section, $(-\Delta)^s$ denotes the Dirichlet spectral realization introduced in Section~\ref{sec2}}. Define the fractional strongly damped wave operator
\[
\mathscr{P}_\beta(z)=\partial_t^2 z + (-\Delta)^{\delta} z + (-\Delta)^\beta \partial_t z.
\] 
The following Carleman estimates are the main tools for establishing conditional stability. We begin with the Carleman estimate for $\partial_t u+(-\Delta)^{\beta}u$.

\begin{lemma}\label{carlmeanfordiff}
{Let $0<\beta\leq1$.}
There exists \( \lambda_0 > 0 \) such that for any \( \lambda \geq \lambda_0 \), there exists \( s_0(\lambda) > 0 \) satisfying: for all \( s > s_0 \) and all solutions \( u \) {corresponding to the spectral Dirichlet realization} and belonging to the regularity space
$u \in C([0, T]; {H_0^\beta(\Omega)}) \cap H^1(0, T; L^2(\Omega)) \cap L^2(0, T; {D(( -\Delta)^\beta)})$,
there exists a constant \( C = C(s_0, \lambda_0,T,\Omega,\beta) > 0 \) such that
\[
\begin{aligned}
& \int_{\Omega_T} \left\{ \frac{1}{s\varphi}\left(|\partial_t u|^2 + |(-\Delta)^\beta u|^2\right) + s\lambda^2\varphi |u|^2 +\lambda \left| (-\Delta)^{\beta/2} u \right|^2 \right\} e^{2s\varphi} \, dxdt  \\
& \leq C \int_{\Omega_T} |\partial_t u + (-\Delta)^\beta u|^2 e^{2s\varphi} \, dxdt + Cs\lambda e^{  2s\varphi(T)}\left( \varphi(T)\|u(\cdot, T)\|_{{H_0^\beta(\Omega)}}^2 + \|u(\cdot, 0)\|_{{H_0^\beta(\Omega)}}^2 \right).
\end{aligned}
\]    
\end{lemma}

The proof of this Lemma is similar to \cite{JIA20181} and \cite{Yamamoto_2009}.

\begin{proof}
{We first perform the calculations for smooth real-valued finite spectral expansions. We assume $s,\lambda\geq1$.}
Let \( L(u) := \partial_t u + (-\Delta)^\beta u \) and define the exponential weight function \( \varphi(t) := e^{\lambda t} \) for a sufficiently large parameter \( \lambda > 0 \). Introduce the conjugate transformation:
\[
v(x,t) := e^{s\varphi(t)} u(x,t),
\]
where \( s > 0 \) is another large parameter to be chosen later. A direct computation yields the conjugated operator:
\[
Pv := e^{s\varphi} L(e^{-s\varphi} v) = \partial_t v - s\lambda \varphi(t) v + (-\Delta)^\beta v = e^{s\varphi} f,
\]
with \( f = L(u) \). {Since the weight depends only on time, $v$ belongs to the same spectral operator domains as $u$.}

By the non-negativity of the square, we expand \( \|Pv\|_{L^2(\Omega_T)}^2 \) as:
\[
\begin{aligned}
\left\| e^{s\varphi} f \right\|_{L^2(\Omega_T)}^2 &= \int_{\Omega_T} |\partial_t v|^2 dxdt + 2\int_{\Omega_T} \partial_t v \left( -s\lambda \varphi v + (-\Delta)^\beta v \right) dxdt \\
&\quad + \int_{\Omega_T} \left| -s\lambda \varphi v + (-\Delta)^\beta v \right|^2 dxdt \\
&\geq \int_{\Omega_T} |\partial_t v|^2 dxdt + 2\int_{\Omega_T} \partial_t v (-\Delta)^\beta v dxdt - 2s\lambda \int_{\Omega_T} \varphi v \partial_t v dxdt.
\end{aligned}
\]
Denote the two cross terms as:
\[
I_1 := 2\int_{\Omega_T} \partial_t v (-\Delta)^\beta v dxdt, \quad I_2 := -2s\lambda \int_{\Omega_T} \varphi v \partial_t v dxdt.
\]
We thus obtain two key inequalities:
\begin{equation}\label{eq:key_ineq1}
\int_{\Omega_T} |f|^2 e^{2s\varphi} dxdt \geq I_1 + I_2,
\end{equation}
\begin{equation}\label{eq:key_ineq2}
\int_{\Omega_T} |\partial_t v|^2 dxdt \leq \int_{\Omega_T} |f|^2 e^{2s\varphi} dxdt + |I_1 + I_2|.
\end{equation}

We first estimate the term \( I_1 \). We use the integration by parts formula for the {spectral Dirichlet fractional Laplacian}
\[
\int_{\Omega} v (-\Delta)^\beta w dx = {\int_\Omega (-\Delta)^{\beta/2}v\,(-\Delta)^{\beta/2}w\,dx}.
\]
Applying this to \( I_1 \):
{
\[
\begin{aligned}
I_1 = \int_0^T \partial_t\left\|(-\Delta)^{\beta/2}v(\cdot,t)\right\|_{L^2(\Omega)}^2\,dt = \left.\left\|(-\Delta)^{\beta/2}v(\cdot,t)\right\|_{L^2(\Omega)}^2\right|_{t=0}^{t=T}.
\end{aligned}
\]
}
Therefore,
\begin{equation}\label{eq:I1_est}
|I_1| \leq C_1 \left( \|v(\cdot,T)\|_{{H_0^\beta(\Omega)}}^2 + \|v(\cdot,0)\|_{{H_0^\beta(\Omega)}}^2 \right),
\end{equation}
where \( C_1 = C_1(n,\beta) > 0 \) is a generic constant independent of \( s \) and \( \lambda \).

Next, we estimate the term \( I_2 \). Integrating \( I_2 \) by parts with respect to time, we obtain

\begin{equation}\label{eq:I2_est}
\begin{aligned}
I_2 &= s\lambda \int_{\Omega_T} v^2 \partial_t \varphi dxdt - s\lambda \left. \int_{\Omega} \varphi v^2 dx \right|_{t=0}^{t=T} \\
&= s\lambda^2 \int_{\Omega_T} \varphi v^2 dxdt - s\lambda \int_{\Omega} \left( \varphi(T) |v(x,T)|^2 - |v(x,0)|^2 \right) dx
\\
&\geq s\lambda^2 \int_{\Omega_T} \varphi v^2 dxdt - s\lambda  \left( \varphi(T)\|v(\cdot,T)\|_{L^2(\Omega)}^2 + \|v(\cdot,0)\|_{L^2(\Omega)}^2 \right).
\end{aligned}
\end{equation}
Substituting \eqref{eq:I1_est} and \eqref{eq:I2_est} into \eqref{eq:key_ineq1}, we obtain:
\begin{equation}\label{eq:prelim_est1}
\begin{aligned}
\left\| f e^{s\varphi} \right\|_{L^2(\Omega_T)}^2 &\geq s\lambda^2 \int_{\Omega_T} \varphi v^2 dxdt - C_1 \left( \|v(\cdot,T)\|_{{H_0^\beta(\Omega)}}^2 + \|v(\cdot,0)\|_{{H_0^\beta(\Omega)}}^2 \right) \\
&\quad - s\lambda  \left( \varphi(T)\|v(\cdot,T)\|_{L^2(\Omega)}^2 + \|v(\cdot,0)\|_{L^2(\Omega)}^2 \right).
\end{aligned}
\end{equation}
We now estimate the fractional semi-norm \( \|(-\Delta)^{\beta/2} v\|_{L^2(\Omega_T)}^2 \). For this, we estimate the following three terms separately.
\[
\int_{\Omega_T} (Pv) v dxdt = \int_{\Omega_T} v \partial_t v dxdt - s\lambda \int_{\Omega_T} \varphi v^2 dxdt + \int_{\Omega_T} v (-\Delta)^\beta v dxdt =: J_1 + J_2 + J_3.
\]
For \( J_1 \):
  \[
  |J_1| = \left| \frac{1}{2} \int_{\Omega_T} \partial_t (v^2) dxdt \right| \leq \frac{1}{2} \left( \|v(\cdot,T)\|_{L^2(\Omega)}^2 + \|v(\cdot,0)\|_{L^2(\Omega)}^2 \right).
  \]
 For \( J_2 \):
  \[
  |J_2| = s\lambda \int_{\Omega_T} \varphi v^2 dxdt.
  \]
 For \( J_3 \) :
  \[
  J_3 = \left\| (-\Delta)^{\beta/2} v \right\|_{L^2(\Omega_T)}^2.
  \]
Substituting these estimates and multiplying both sides by \( \lambda \):
\begin{equation}\label{eq:sem_norm_pre}
\begin{aligned}
\lambda \left\| (-\Delta)^{\beta/2} v \right\|_{L^2(\Omega_T)}^2 &\leq \lambda \int_{\Omega_T} |Pv| |v| dxdt + C_2 s\lambda^2 \int_{\Omega_T} \varphi v^2 dxdt \\
&\quad + C_2 \lambda \left( \|v(\cdot,T)\|_{L^2(\Omega)}^2 + \|v(\cdot,0)\|_{L^2(\Omega)}^2 \right).
\end{aligned}
\end{equation}
Apply Young's inequality \( ab \leq \frac{1}{2}a^2 + \frac{1}{2}b^2 \) to the first term on the right-hand side of \eqref{eq:sem_norm_pre}:
\begin{equation}\label{eq:young_ineq}
\lambda \int_{\Omega_T} |Pv| |v| dxdt \leq \frac{1}{2} \|Pv\|_{L^2(\Omega_T)}^2 + \frac{\lambda^2}{2} \|v\|_{L^2(\Omega_T)}^2 = \frac{1}{2} \left\| f e^{s\varphi} \right\|_{L^2(\Omega_T)}^2 + \frac{\lambda^2}{2} \|v\|_{L^2(\Omega_T)}^2.
\end{equation}
Substituting \eqref{eq:young_ineq} into \eqref{eq:sem_norm_pre} and employing \eqref{eq:prelim_est1} to estimate the \( s\lambda^2 \int \varphi v^2 dxdt \) term, we obtain the inequality

\begin{equation}\label{eq:sem_norm_est}
\begin{aligned}
\lambda \left\| (-\Delta)^{\beta/2} v \right\|_{L^2(\Omega_T)}^2 &\leq C_3 \left\| f e^{s\varphi} \right\|_{L^2(\Omega_T)}^2 + C_3 \lambda^2 \|v\|_{L^2(\Omega_T)}^2 \\
&\quad + C_3 s\lambda  \left( \varphi(T)\|v(\cdot,T)\|_{{H_0^\beta(\Omega)}}^2 + \|v(\cdot,0)\|_{{H_0^\beta(\Omega)}}^2 \right).
\end{aligned}
\end{equation}

Combine \eqref{eq:prelim_est1} and \eqref{eq:sem_norm_est}. Choose \( \lambda \geq \lambda_0 \) (sufficiently large) and then \( s > s_0(\lambda) \) (sufficiently large) such that the term \( C_3 \lambda^2 \|v\|_{L^2(\Omega_T)}^2 \) is absorbed by the left-hand side term \( s\lambda^2 \int_{\Omega_T} \varphi v^2 dxdt \). This yields the core estimate:

\begin{equation}\label{eq:core_est_v}
\begin{aligned}
& s\lambda^2 \int_{\Omega_T} \varphi v^2 dxdt + \lambda \left\| (-\Delta)^{\beta/2} v \right\|_{L^2(\Omega_T)}^2 \\
&\leq C_4 \left\| f e^{s\varphi} \right\|_{L^2(\Omega_T)}^2 + C_4 s\lambda  \left(\varphi(T)\|v(\cdot,T)\|_{{H_0^\beta(\Omega)}}^2 + \|v(\cdot,0)\|_{{H_0^\beta(\Omega)}}^2 \right).
\end{aligned}
\end{equation}
Since \( v = e^{s\varphi} u \), substitute into \eqref{eq:core_est_v} to transform back to the original variable \( u \):

\begin{equation}\label{eq:core_est_u1}
\begin{aligned}
& s\lambda^2 \int_{\Omega_T} \varphi |u|^2 e^{2s\varphi} dxdt + \lambda \int_0^T \left\| (-\Delta)^{\beta/2} u \right\|_{L^2(\Omega)}^2 e^{2s\varphi} dt \\
\leq& C_4 \int_{\Omega_T} |f|^2 e^{2s\varphi} dxdt + C_4s\lambda e^{2s\varphi(T)} \left( {\varphi(T)}\|u(\cdot,T)\|_{{H_0^\beta(\Omega)}}^2 + \|u(\cdot,0)\|_{{H_0^\beta(\Omega)}}^2 \right).
\end{aligned}
\end{equation}

Next, estimate the time derivative term. From \( \partial_t u = e^{-s\varphi} (\partial_t v - s\lambda \varphi v) \), we have:
\[
\frac{1}{s\varphi} |\partial_t u|^2 e^{2s\varphi} \leq 2s\lambda^2 \varphi v^2 + \frac{2}{s\varphi} |\partial_t v|^2.
\]
{
The square expansion and \eqref{eq:I1_est}--\eqref{eq:I2_est} give
\[
\begin{aligned}
\int_{\Omega_T}|\partial_t v|^2\,dxdt
\leq{}&\int_{\Omega_T}|f|^2e^{2s\varphi}\,dxdt
+C_1\left(\|v(\cdot,T)\|_{H_0^\beta(\Omega)}^2+\|v(\cdot,0)\|_{H_0^\beta(\Omega)}^2\right)\\
&+s\lambda\left(\varphi(T)\|v(\cdot,T)\|_{L^2(\Omega)}^2+\|v(\cdot,0)\|_{L^2(\Omega)}^2\right).
\end{aligned}
\]
Combining this estimate with \eqref{eq:core_est_u1} and using $s,\lambda,\varphi\geq1$, we get
}
\begin{equation}\label{eq:dt_est}
\begin{aligned}
\int_{\Omega_T} \frac{1}{s\varphi} |\partial_t u|^2 e^{2s\varphi} dxdt 
\leq& C_5 \int_{\Omega_T} |f|^2 e^{2s\varphi} dxdt \\
&+ C_5 s\lambda e^{2s\varphi(T)} \left( \varphi(T)\|u(\cdot,T)\|_{{H_0^\beta(\Omega)}}^2 + \|u(\cdot,0)\|_{{H_0^\beta(\Omega)}}^2 \right).
\end{aligned}
\end{equation}

Finally, by the operator definition \( (-\Delta)^\beta u = f - \partial_t u \), apply the triangle inequality:
\[
\frac{1}{s\varphi} |(-\Delta)^\beta u|^2 e^{2s\varphi} \leq \frac{2}{s\varphi} |f|^2 e^{2s\varphi} + \frac{2}{s\varphi} |\partial_t u|^2 e^{2s\varphi}.
\]
Since \( s > s_0 > 1 \), the first term \( \frac{2}{s\varphi} |f|^2 e^{2s\varphi} \leq 2 |f|^2 e^{2s\varphi} \) can be absorbed by the right-hand side integral of \( |f|^2 e^{2s\varphi} \). Combining with \eqref{eq:dt_est}, we see that
\begin{equation}\label{eq:flap_est}
\begin{aligned}
\int_{\Omega_T} \frac{1}{s\varphi} |(-\Delta)^\beta u|^2 e^{2s\varphi} dxdt 
\leq& C_6 \int_{\Omega_T} |f|^2 e^{2s\varphi} dxdt \\
&+ C_6 s\lambda e^{ 2s\varphi(T)} \left( \varphi(T)\|u(\cdot,T)\|_{{H_0^\beta(\Omega)}}^2 + \|u(\cdot,0)\|_{{H_0^\beta(\Omega)}}^2 \right).
\end{aligned}
\end{equation}

Merge \eqref{eq:core_est_u1}, \eqref{eq:dt_est} and \eqref{eq:flap_est}. Note that the coefficient \( \frac{\lambda}{2} \) in the lemma statement is absorbed into the generic constant \( C \). We finally obtain the desired inequality, which completes the proof of the lemma.
\end{proof}

\begin{theorem}\label{thm1}
Let ${0<\beta\leq1,\quad 0\leq\delta\leq\beta}$ and let $z \in C([0,T]; {D(( -\Delta)^\beta)}) \cap H^2(0,T; L^2(\Omega)) \cap {H^1(0,T;D(( -\Delta)^\beta))}$ with $z(\cdot,0) = 0$ and $\partial_t z(\cdot,0)=0$ in $\Omega$.
Then there exists $\lambda_0 > 0$ such that for all $\lambda \ge \lambda_0$, there exist ${s_0(\lambda)>0}$ and $C = C(s_0, \lambda_0, T, \Omega, \beta{,\delta}) > 0$ such that {for all $s>s_0(\lambda)$}
\begin{equation*}
    \begin{aligned}
& \int_{\Omega_T} \Bigl( \frac{1}{s\varphi}  |\partial_t^2 z|^2 + s\lambda|(-\Delta)^\beta z|^2 + \lambda |(-\Delta)^{\beta/2} \partial_t z|^2 + s^3\lambda^4|z|^2 + s\lambda^2 \varphi|\partial_t z|^2 \Bigr) e^{2s\varphi} \, dx dt \\
\leq& C\left( \int_{\Omega_T} |\mathscr{P}_\beta(z)|^2 e^{2s\varphi} \, dx dt + s^2\lambda^3 e^{2s\varphi(T)}\left( \|z(\cdot, T)\|_{{D(( -\Delta)^\beta)}}^2 + {\varphi(T)}\|\partial_t z(\cdot, T)\|_{{H_0^\beta(\Omega)}}^2\right)\right).
\end{aligned}   
\end{equation*}
\end{theorem}

\begin{proof}
Firstly, we obtain the following estimate when ${z}(x,0)=0$ for the Carleman estimate of Lemma \ref{carlmeanfordiff}:
\begin{equation*}
    \begin{aligned}
& \int_{\Omega_T} \Bigl( \frac{1}{s\varphi} \left( |\partial_t z|^2 + |(-\Delta)^\beta z|^2 \right) + \frac{\lambda}{2} |(-\Delta)^{\beta/2} z|^2 + s\lambda^2 \varphi z^2 \Bigr) e^{2s\varphi} \, dx dt \\
\leq& C \int_{\Omega_T} |\partial_t z + (-\Delta)^\beta z|^2 e^{2s\varphi} \, dx dt + Cs\lambda e^{2s\varphi(T)}\varphi(T) \|z(\cdot, T)\|_{{H_0^\beta(\Omega)}}^2.
\end{aligned}   
\end{equation*}
Since $\partial_t^2z + (-\Delta)^\beta \partial_t z = \mathscr{P}_\beta(z) - (-\Delta)^{\delta} z$. {The assumed regularity and $\partial_t z(\cdot,0)=0$ allow us to apply Lemma~\ref{carlmeanfordiff} to $\partial_t z$.} Inspired by the skill in \cite{wubin2012}, we have the following Carleman estimate:
\begin{equation}\label{wholecarleman}
    \begin{aligned}
& \int_{\Omega_T} \Bigl( \frac{1}{s\varphi} \left( |\partial_t^2 z|^2 + |(-\Delta)^\beta \partial_t z|^2 \right) + \frac{\lambda}{2} |(-\Delta)^{\beta/2} \partial_t z|^2 + s\lambda^2 \varphi|\partial_t z|^2 \Bigr) e^{2s\varphi} \, dx dt \\
\leq& C \int_{\Omega_T} |\mathscr{P}_\beta(z)|^2 e^{2s\varphi} \, dx dt + {C}\int_{\Omega_T}e^{2s\varphi}|(-\Delta)^{\delta} z|^2\,dxdt + {C}s\lambda e^{2s\varphi(T)} \varphi(T)\|\partial_t z(\cdot, T)\|_{{H_0^\beta(\Omega)}}^2.
\end{aligned}   
\end{equation}
Using \eqref{remarkeq} in Remark \ref{remark1} with $w=\partial_t z$ , we have
\[
\begin{aligned}
\int_{\Omega_T} e^{2s \varphi}\Bigl|\int_0^t \partial_t z(\cdot,\tau)d\tau\Bigr|^{2}\,dxdt
\le& \frac{e^{2s \varphi(T)}}{s\lambda}\int_\Omega\Bigl|\int_0^T \partial_t z(\cdot,\tau)d\tau\Bigr|^{2}\,dx
\\
& + \frac{1}{s^{2}\lambda^{2}}\int_{\Omega_T} e^{2s\varphi}|\partial_t z(\cdot,t)|^{2}\,dxdt.
\end{aligned}
\]
In view of the fact that $z(x,0)=0$ for $x\in \Omega$, we obtain
\begin{equation}\label{aply_estimate}
\begin{aligned}
s^3\lambda^4 \int_{\Omega_T} e^{2s \varphi}| z(\cdot,t)|^{2}\,dxdt
\le s^2\lambda^3 e^{2s \varphi(T)}\|z(\cdot,T)\|_{L^2(\Omega)}^{2}   + s\lambda^2 \int_{\Omega_T} e^{2s\varphi}|\partial_t z(\cdot,t)|^{2}\,dxdt.
\end{aligned}
\end{equation}
{Since $(-\Delta)^\beta z\in H^1(0,T;L^2(\Omega))$ and $(-\Delta)^\beta z(\cdot,0)=0$, }using Lemma \ref{lem:terminal}, we have
\begin{equation*}
    \begin{aligned}
 \int_{\Omega_T} \varphi(t) e^{2s \varphi}\Bigl|\int_0^t \partial_t (-\Delta)^\beta z(\cdot,\tau)d\tau\Bigr|^{2}\,dxdt &\le
 \frac{e^{2s \varphi(T)}}{s\lambda}\int_\Omega\Bigl|\int_0^T \partial_t (-\Delta)^\beta z(\cdot,\tau)d\tau\Bigr|^{2}\,dx
   \\
   &\quad+\frac{1}{s^{2}\lambda^{2}}\int_{\Omega_T} \frac{e^{2s\varphi}}{\varphi}|\partial_t (-\Delta)^\beta z(\cdot,t)|^{2}\,dxdt,
    \end{aligned}
\end{equation*}
which further yields
\begin{equation}\label{estimate_Deltau}
\begin{aligned}
     s\lambda\int_{\Omega_T} e^{2s \varphi}\bigl| (-\Delta)^\beta z(\cdot,t)\bigr|^{2}\,dxdt
\le& e^{2s\varphi(T)}\|z(\cdot,T)\|_{{D(( -\Delta)^\beta)}}^2
\\
   &+ \frac{1}{s\lambda}\int_{\Omega_T} \frac{e^{2s\varphi}}{\varphi}|\partial_t(-\Delta)^{\beta}z(\cdot,t)|^{2}\,dxdt.
\end{aligned}
\end{equation}

Combining \eqref{wholecarleman},\eqref{aply_estimate} and \eqref{estimate_Deltau} with $\lambda \geq {\lambda_0}$, we get
\begin{equation*}
    \begin{aligned}
& \int_{\Omega_T} \Bigl( \frac{1}{s\varphi}  |\partial_t^2 z|^2 + s\lambda|(-\Delta)^\beta z|^2 + \lambda |(-\Delta)^{\beta/2} \partial_t z|^2 + s^3\lambda^4|z|^2 + s\lambda^2 \varphi|\partial_t z|^2 \Bigr) e^{2s\varphi} \, dx dt \\
\le& C\left( \int_{\Omega_T} |\mathscr{P}_\beta(z)|^2 e^{2s\varphi} \, dx dt + \int_{\Omega_T}e^{2s\varphi}|(-\Delta)^{\delta} z|^2\,dxdt\right.\\
&\left.+ s^2\lambda^3 e^{2s\varphi(T)}\left( \|z(\cdot, T)\|_{{D(( -\Delta)^\beta)}}^2 +\varphi(T)\|\partial_t z(\cdot, T)\|_{{H_0^\beta(\Omega)}}^2\right)\right).
\end{aligned}   
\end{equation*}
Finally, by noting the interpolation estimate (see \cite{bergh1976interpolation})
\begin{equation*}
    \|(-\Delta)^{\delta} z\|_{L^2(\Omega)} \le C(\beta)(\|z\|_{L^2(\Omega)}+\|(-\Delta)^{\beta} z\|_{L^2(\Omega)}), \quad
     \delta \leq \beta
\end{equation*}
{Applying this estimate to $e^{s\varphi}z$ and integrating by time, we get

\begin{equation}\label{eq_absorb}
\int_{\Omega_T}e^{2s\varphi}|(-\Delta)^{\delta} z|^2dx dt \leq \int_{\Omega_T}e^{2s\varphi}|z|^2dxdt+\int_{\Omega_T}e^{2s\varphi}|(-\Delta)^{\beta} z|^2dxdt
\end{equation}
}
Choosing $s\geq s_1$, we obtain
\begin{equation*}
    \begin{aligned}
& \int_{\Omega_T} \Bigl( \frac{1}{s\varphi}  |\partial_t^2 z|^2 + s\lambda|(-\Delta)^\beta z|^2 + \lambda |(-\Delta)^{\beta/2} \partial_t z|^2 + s^3\lambda^4|z|^2 + s\lambda^2 \varphi|\partial_t z|^2 \Bigr) e^{2s\varphi} \, dx dt \\
\leq& C\left( \int_{\Omega_T} |\mathscr{P}_\beta(z)|^2 e^{2s\varphi} \, dx dt + s^2\lambda^3 e^{2s\varphi(T)}\left( \|z(\cdot, T)\|_{{D(( -\Delta)^\beta)}}^2 +\varphi(T)\|\partial_t z(\cdot, T)\|_{{H_0^\beta(\Omega)}}^2\right)\right).
\end{aligned}   
\end{equation*}
This completes the proof of the theorem.
\end{proof}

{ Next, we introduce a Carleman estimate for Riesz definition. We define
\[
\mathscr P_\beta^{\mathrm R}(z)
=\partial_t^2z+\mathcal L_\delta z
                   +\mathcal L_\beta\partial_tz.
\]
\begin{corollary}
\label{cor:riesz_carleman}
Let $0<\beta<1$ and $0\leq\delta\leq\beta$. Assume in addition that
\[
\delta=\beta
\quad\text{or}\quad
2\delta<\beta+\frac12.
\]
Let $\mathcal L_s$ and $\mathcal H_0^s(\Omega)$ be the Riesz
realization and its energy space defined in Section~\ref{sec2},
and set $\mathcal L_0=I$. Suppose that $z\in H^2(0,T;L^2(\Omega)) \cap H^1(0,T;D(\mathcal L_\beta))$, $z(0)=\partial_tz(0)=0
$. Then there exist $\lambda_0\geq1$ and $C>0$ such that, for every
$\lambda\geq\lambda_0$, there exists $s_0(\lambda)\geq1$ for which
the following estimate holds for all $s\geq s_0(\lambda)$:
\[
\begin{aligned}
&\int_0^T\Bigl\{
\frac{\|\partial_t^2z\|_{L^2(\Omega)}^2}{s\varphi}
+s\lambda\|\mathcal L_\beta z\|_{L^2(\Omega)}^2
+\lambda\|\mathcal L_\beta^{1/2}\partial_tz\|_{L^2(\Omega)}^2
+s^3\lambda^4\|z\|_{L^2(\Omega)}^2
+s\lambda^2\varphi\|\partial_tz\|_{L^2(\Omega)}^2
\Bigr\}e^{2s\varphi}\,dt\\
&\quad\leq C\Biggl\{
\int_0^T\|\mathscr P_\beta^{\mathrm R}(z)\|_{L^2(\Omega)}^2
e^{2s\varphi}\,dt
+s^2\lambda^3e^{2s\varphi(T)}
\left(
\|z(T)\|_{D(\mathcal L_\beta)}^2
+\varphi(T)\|\partial_tz(T)\|_{\mathcal H_0^\beta(\Omega)}^2
\right)\Biggr\}.
\end{aligned}
\]
The constant $C$ is independent of $s$, $\lambda$, and $z$, and may
depend on $T$, $\Omega$, $\beta$, and $\delta$.
\end{corollary}

\begin{proof}
The zero-exterior Riesz realization is positive and self-adjoint,
and its form domain is $\mathcal H_0^\beta(\Omega)$
\cite{Grubb2015FractionalDomains,Grubb2018FractionalOperators}.
For $v\in D(\mathcal L_\beta)$,
\[
(\mathcal L_\beta v,v)_{L^2(\Omega)}
=\|\mathcal L_\beta^{1/2}v\|_{L^2(\Omega)}^2
=\|(-\Delta)_{\mathbb R^n}^{\beta/2}\widetilde v
\|_{L^2(\mathbb R^n)}^2.
\]
Accordingly, the proof of Lemma~\ref{carlmeanfordiff} applies to
$\partial_t+\mathcal L_\beta$, with
$\mathcal L_\beta^{1/2}$ in the energy term and
$\mathcal H_0^\beta(\Omega)$ in the endpoint norms.

For $0<\beta<\frac12$, the domain regularity estimate gives
\begin{equation}
\|\widetilde v\|_{H^{2\beta}(\mathbb R^n)}
\leq C\|v\|_{D(\mathcal L_\beta)}.
\end{equation}
For $\frac12\leq\beta<1$ and every sufficiently small
$\varepsilon>0$, it gives
\begin{equation}
\|\widetilde v\|_{H^{\beta+1/2-\varepsilon}(\mathbb R^n)}
\leq C_\varepsilon\|v\|_{D(\mathcal L_\beta)}.
\end{equation}
These follow from the domain characterization and embeddings
reviewed in \cite{Grubb2018FractionalOperators}. If $\delta<\beta$ satisfies the stated assumptions, these estimates
and the full-space Fourier multiplier bound imply
\[
\begin{aligned}
\|\mathcal L_\delta v\|_{L^2(\Omega)}\leq C\|\widetilde v\|_{H^{2\delta}(\mathbb R^n)}\leq C_{\delta,\beta}\|v\|_{D(\mathcal L_\beta)}.
\end{aligned}
\]
For $\delta=\beta$, it follows directly from the definition of the
graph norm. Thus, in all cases covered by the corollary,
\[
\|\mathcal L_\delta v\|_{L^2(\Omega)}^2
\leq C_{\delta,\beta}
\left(
\|v\|_{L^2(\Omega)}^2
+\|\mathcal L_\beta v\|_{L^2(\Omega)}^2
\right).
\]
Repeating the proof of Theorem~\ref{thm1} with
$\mathcal L_\beta$ and $\mathcal L_\delta$ gives the corresponding estimate \eqref{eq_absorb} before absorption. Choosing $s_0(\lambda)$ sufficiently large
absorbs the two resulting terms into
$s^3\lambda^4\|z\|_{L^2(\Omega)}^2$ and
$s\lambda\|\mathcal L_\beta z\|_{L^2(\Omega)}^2$, respectively.
This proves the corollary.
\end{proof}
}

\section{Proof of the Conditional Stability}\label{sec4}
In this section, we prove the conditional stability estimates by virtue of the Carleman estimates established in Section \ref{sec3}.

\begin{proof}[\bf Proof of Theorem \ref{thm3}]
{Assume $u\in H^2(0,T;L^2(\Omega))\allowbreak\cap H^1(0,T;D(( -\Delta)^\beta))$ for the spectral realization.}
To make sure the initial displacement is equal to $0$, we introduce the following cut-off function which also important to the conditional stability. We fix $\lambda>0$ and choose $t_1,t_2$ such that $0<t_1<t_2<t_0$. Define $\varphi(t)=e^{\lambda t}$.  
Let $\chi\in C^\infty(\mathbb{R})$ be a cut‑off function with $0\le\chi\le1$ and  
\begin{equation}\label{cutoff}
    \chi(t)=\begin{cases}
1, & t>t_2,\\
0, & t<t_1.
\end{cases}
\end{equation}

Set $w=\chi u$. Then $\mathscr{P}_\beta(w)=\chi''(t)u+2\chi'(t)\partial_t u+\chi'(t)(-\Delta)^\beta u{+a\chi'(t)(-\Delta)^{\gamma/2}u}- a\partial_t(-\Delta)^{\frac{\gamma}{2}}w$ and set $\mathscr{P}(u)=\chi''(t)u+2\chi'(t)\partial_t u+\chi'(t)(-\Delta)^\beta u{+a\chi'(t)(-\Delta)^{\gamma/2}u}$.Then by the Cauchy inequality,
\[
|\mathscr{P}_\beta(w)|^2 \le  2|\mathscr{P}(u)|^2+ 2| a|^2 |\partial_t(-\Delta)^{\gamma/2}w|^2 .
\]
Since $\chi{,\chi',\chi''}$ and $a$ are bounded, the constants can be absorbed into the generic constant $C$.

Applying the Carleman estimate in Theorem \ref{thm1} for $w$, we obtain
\begin{equation}
\begin{aligned}
    & \int_{\Omega_T} e^{2s \varphi}\left( \frac{1}{s\varphi}  |\partial_t^2 w|^2+s\lambda^2\varphi|\partial_t w|^2 \ + {\lambda}|(-\Delta)^{\beta/2} \partial_t w|^2+s^3\lambda^4|w|^2 \right) \, dxdt \\
    \leq& C \left(\int_{\Omega_T} e^{2s \varphi} | \mathscr{P}(u)|^2\, dxdt \right.+\int_{\Omega_T}e^{2s\varphi}|\partial_t(-\Delta)^{\frac{\gamma}{2}}w|^2 \,dxdt\\
    &\quad + s^2\lambda^3 e^{2s \varphi(T)}\left(\varphi(T)\|\partial_t w(\cdot,T)\|_{{H_0^\beta(\Omega)}}^2+\| w(\cdot,T)\|_{{D(( -\Delta)^\beta)}}^2\right)\Biggr).
\end{aligned}
\end{equation}

Now we need to estimate $\partial_t(-\Delta)^{\frac{\gamma}{2}} {w}$. It is not difficult to check that 
\begin{equation*}
    \|\partial_t(-\Delta)^{\frac{\gamma}{2}} w\|_{L^2(\Omega)} \leq C(\beta)(\|\partial_t w\|_{L^2(\Omega)}+\|\partial_t(-\Delta)^{\frac{\beta}{2}} w\|_{L^2(\Omega)}),\quad \gamma \leq \beta,
\end{equation*}
{Repeating the same argument in \eqref{eq_absorb}}. Then, by choosing $\lambda \geq {\lambda_*}$, we can absorb all the space fractional damping terms into the left‑hand side and thus we obtain 
\begin{equation}\label{4.3}
    \begin{aligned}
& \int_{\Omega_T} e^{2s\varphi}\bigl(\frac{1}{s\varphi}  |\partial_t^2 w|^2+s\lambda^2\varphi|\partial_t w|^2+\frac{\lambda}{2}|(-\Delta)^{\beta/2}\partial_t w|^2+s^3\lambda^4|w|^2\bigr)\,dxdt\\
\le& C\Biggl(
\int_{t_1}^{t_2}\int_\Omega e^{2s\varphi}|\mathscr{P}(u)|^2\,dxdt+ s^2\lambda^3 e^{2s\varphi(T)}\bigl(\varphi(T)\|\partial_t w(\cdot,T)\|_{{H_0^\beta(\Omega)}}^2+\|w(\cdot,T)\|_{{D(( -\Delta)^\beta)}}^2\bigr)\Biggr).
\end{aligned}
\end{equation}

Since {$\chi'$ and $\chi''$} are supported in $[t_1,t_2]$, we have $|\mathscr{P}(u)|\le C_0(|u|+|\partial_t u|+|(-\Delta)^\beta u|\allowbreak{+|(-\Delta)^{\gamma/2}u|})$ on $(t_1,t_2)$.  
Using the a priori bound $\|u\|_{H^1(0,T;{D(( -\Delta)^\beta)})}\le M$ and the fact that $e^{2s\varphi}\le e^{2s\varphi(t_2)}$ on $(t_1,t_2)$, we obtain  
\[
\int_{t_1}^{t_2}\int_\Omega e^{2s\varphi}|\mathscr{P}(u)|^2\,dxdt\le C_1 e^{2s\varphi(t_2)}M^2.
\]  
Moreover, because $w=u$ for $t>t_2$, we have 
{
\[
\|\partial_t w(\cdot,T)\|_{{H_0^\beta(\Omega)}}^2+\|w(\cdot,T)\|_{{D(( -\Delta)^\beta)}}^2\le C_2\bigl(\|u(\cdot,T)\|_{{D(( -\Delta)^\beta)}}^2+\|\partial_t u(\cdot,T)\|_{{H_0^\beta(\Omega)}}^2\bigr).
\]  
}
Thus \eqref{4.3} becomes
\begin{equation}\label{eq5.191}
\begin{aligned}
& \int_{\Omega_T} e^{2s\varphi}\bigl(\frac{1}{s\varphi}  |\partial_t^2 w|^2+s\lambda^2\varphi|\partial_t w|^2+\frac{\lambda}{2}|(-\Delta)^{\beta/2}\partial_t w|^2+s^3\lambda^4|w|^2\bigr)\,dxdt\\
&\le C_3\left(e^{2s\varphi(t_2)}M^2 + s^2\lambda^3 e^{2s\varphi(T)}\bigl(\|u(\cdot,T)\|_{{D(( -\Delta)^\beta)}}^2+\varphi(T)\|\partial_t u(\cdot,T)\|_{{H_0^\beta(\Omega)}}^2\bigr)\right).
\end{aligned}
\end{equation}

For $t\ge t_0$ we have $w=u$, and $e^{2s\varphi}\ge e^{2s\varphi(t_0)}$. Hence  
\[
\begin{aligned}
&e^{2s\varphi(t_0)}\int_{t_0}^T\int_\Omega\bigl(\frac{1}{s\varphi}  |\partial_t^2 u|^2+s\lambda\varphi|\partial_t u|^2+s^3\lambda^4|u|^2\bigr)dxdt\\
\le& \int_{\Omega_T} e^{2s\varphi}\bigl({\frac{1}{s\varphi}|\partial_t^2w|^2+s\lambda^2\varphi|\partial_t w|^2+s^3\lambda^4|w|^2}\bigr)dxdt.
\end{aligned}
\]  
Choosing a suitable $C_4(\lambda)$ depending on $\lambda$, and  combining the above estimates yields  
\[
\begin{aligned}
&\int_{t_0}^T\int_\Omega\bigl(\frac{1}{s\varphi}  |\partial_t^2 u|^2+s\lambda\varphi|\partial_t u|^2+s^3\lambda^4|u|^2\bigr)dxdt
\\
\le & C_4 e^{-2s(\varphi(t_0)-\varphi(t_2))}M^2 
+ C_4(\lambda) s^2\lambda^3 e^{2s(\varphi(T)-\varphi(t_0))}\bigl(\|u(\cdot,T)\|_{{D(( -\Delta)^\beta)}}^2+\|\partial_t u(\cdot,T)\|_{{H_0^\beta(\Omega)}}^2\bigr).
\end{aligned}
\]

Given the relation among $t_0,t_1,t_2$, we denote
\[
\alpha_1:=\varphi(t_0)-\varphi(t_2)>0,\qquad 
\rho:=\varphi(T)-\varphi(t_0)>0,
\]  
and set  
\[
\mathcal{D}^2:=\|u(\cdot,T)\|_{{D(( -\Delta)^\beta)}}^2+\|\partial_t u(\cdot,T)\|_{{H_0^\beta(\Omega)}}^2.
\]  
Then \eqref{eq5.191} becomes  
\[
\int_{t_0}^T\int_\Omega\bigl(\frac{1}{s\varphi}  |\partial_t^2 u|^2+s\lambda\varphi|\partial_t u|^2+s^3\lambda^4|u|^2\bigr)dxdt
\le C_4 e^{-2s\alpha_1}M^2 + C_4 s^2\lambda^3 e^{2s\rho} \mathcal{D}^2. 
\]

Now we estimate $\|u(\cdot,t_0)\|_{L^2(\Omega)}^2+\|\partial_t u(\cdot,t_0)\|_{L^2(\Omega)}^2$:

\[
\begin{aligned}
&\|u(\cdot,t_0)\|_{L^2(\Omega)}^2 + \|\partial_t u(\cdot,t_0)\|_{L^2(\Omega)}^2 \\
=& -\int_{t_0}^T \partial_t \left( \|u(t)\|_{L^2(\Omega)}^2 + \|\partial_t u(t)\|_{L^2(\Omega)}^2 \right) dt + \|u(\cdot,T)\|_{L^2(\Omega)}^2 + \|\partial_t u(\cdot,T)\|_{L^2(\Omega)}^2 \\
=& -2\int_{t_0}^T \int_\Omega \left( u\,\partial_t u + \partial_t^2 u\,\partial_t u \right) dxdt + \|u(\cdot,T)\|_{L^2(\Omega)}^2 + \|\partial_t u(\cdot,T)\|_{L^2(\Omega)}^2.
\end{aligned}
\]
Using the Cauchy–Schwarz inequality and the Young inequality, the cross terms can be controlled. For a parameter $s>0$, we have
\[
-2u\,\partial_t u \le s^3 |u|^2 + \frac{1}{s^3}|\partial_t u|^2,\qquad
-2\partial_t u\,\partial_t^2 u \le s\varphi|\partial_t u|^2 + \frac{1}{s\varphi}|\partial_t^2 u|^2.
\]
Consequently,
\[
\begin{aligned}
&\|u(\cdot,t_0)\|_{L^2(\Omega)}^2 + \|\partial_t u(\cdot,t_0)\|_{L^2(\Omega)}^2 \\
\le& \int_{t_0}^T \int_\Omega \left( s^3|u|^2 + \bigl(s\varphi + \tfrac{1}{s^3}\bigr)|\partial_t u|^2 + \tfrac{1}{s\varphi}|\partial_t^2 u|^2 \right) dxdt
+ \|u(\cdot,T)\|_{L^2(\Omega)}^2 + \|\partial_t u(\cdot,T)\|_{L^2(\Omega)}^2.
\end{aligned}
\]
Choosing suitable $s,\lambda>0$, we easily obtain
\[
\|u(\cdot,t_0)\|_{L^2(\Omega)}^2 +\|\partial_t u(\cdot,t_0)\|_{L^2(\Omega)}^2\le C_4 e^{-2s\alpha_1}M^2 + C_4(\lambda) s^2\lambda^3 e^{2s\rho}\mathcal{D}^2.
\]  
{Fix $\kappa>0$. Since $\lambda$ is fixed, $s^2\lambda^3\leq C_{\kappa,\lambda}e^{2s\kappa}$ for $s\geq s_0$, after increasing $s_0$ if necessary.} We choose a large $C_5$, and then we obtain
\begin{equation}\label{estimate_1}
    \|u(\cdot,t_0)\|_{L^2(\Omega)}^2 +\|\partial_t u(\cdot,t_0)\|_{L^2(\Omega)}^2 \le C_4 e^{-2s\alpha_1}M^2 + C_5 e^{2s(\rho{+\kappa})}\mathcal{D}^2.
\end{equation}

First, let $\mathcal{D} = 0$. Then letting $s \to \infty$, we see that $u(x,t_0) = 0$ {and $\partial_tu(x,t_0)=0$.} Thus the conclusion holds.
Second, assume $\mathcal{D} \neq 0$. {If $\mathcal{D}\geq M$, the equation and the graph-norm a priori bound give $\|\partial_t^2u\|_{L^2(\Omega_T)}\leq CM$. The time trace theorem therefore yields $\|u(\cdot,t_0)\|_{L^2(\Omega)}+\|\partial_tu(\cdot,t_0)\|_{L^2(\Omega)}\leq CM$, which gives the desired estimate.} Now let $\mathcal{D} \leq M$. We choose $s>0$ {to balance the two terms on the right-hand side} of \eqref{estimate_1} and obtain an estimate, namely,
\[
s = \frac{1}{2(\alpha_1+\rho{+\kappa})}\ln\left(\frac{M^2}{\mathcal{D}^2}\right) + s_0.
\]  
Then  
\[
e^{-2s\alpha_1} =  e^{-\frac{\alpha_1}{\alpha_1+\rho{+\kappa}}\ln\frac{M^2}{\mathcal{D}^2}} e^{-2s_0\alpha_1} = e^{-2s_0\alpha_1} \left(\frac{\mathcal{D}^2}{M^2}\right)^{\frac{\alpha_1}{\alpha_1+\rho{+\kappa}}},
\]  
\[
e^{2s(\rho{+\kappa})} =  e^{\frac{\rho{+\kappa}}{\alpha_1+\rho{+\kappa}}\ln\frac{M^2}{\mathcal{D}^2}} e^{2s_0(\rho{+\kappa})} = e^{2s_0(\rho{+\kappa})} \left(\frac{M^2}{\mathcal{D}^2}\right)^{\frac{\rho{+\kappa}}{\alpha_1+\rho{+\kappa}}}.
\]  
Substituting into \eqref{estimate_1} gives  
\[
\|u(\cdot,t_0)\|_{L^2(\Omega)}^2+\|\partial_t u(\cdot,t_0)\|_{L^2(\Omega)}^2 \le C\Bigl(  M^2 \bigl(\tfrac{\mathcal{D}^2}{M^2}\bigr)^{\frac{\alpha_1}{\alpha_1+\rho{+\kappa}}} +  \mathcal{D}^2 \bigl(\tfrac{M^2}{\mathcal{D}^2}\bigr)^{\frac{\rho{+\kappa}}{\alpha_1+\rho{+\kappa}}} \Bigr).
\]  
Thus there exists a constant $C_6$ such that  
\[
\|u(\cdot,t_0)\|_{L^2(\Omega)}^2+\|\partial_t u(\cdot,t_0)\|_{L^2(\Omega)}^2 \le C_6 M^{\frac{2(\rho{+\kappa})}{\alpha_1+\rho{+\kappa}}} \mathcal{D}^{\frac{2\alpha_1}{\alpha_1+\rho{+\kappa}}}.
\]  
Taking square roots and setting $\theta_1 = \frac{\alpha_1}{\alpha_1+\rho{+\kappa}}\in(0,1)$ yields  
\[
\|u(\cdot,t_0)\|_{L^2(\Omega)}+\|\partial_t u(\cdot,t_0)\|_{L^2(\Omega)} \le C M^{1-\theta_1} \mathcal{D}^{\theta_1}.
\]  

Finally, $\mathcal{D} = \bigl(\|u(\cdot,T)\|_{{D(( -\Delta)^\beta)}}^2+\|\partial_t u(\cdot,T)\|_{{H_0^\beta(\Omega)}}^2\bigr)^{1/2} \le \|u(\cdot,T)\|_{{D(( -\Delta)^\beta)}}+\|\partial_t u(\cdot,T)\|_{{H_0^\beta(\Omega)}}$. This completes the proof of the theorem.
\end{proof}

\begin{remark}\label{remark4.1}
The exponent $\theta_1$ in the conditional stability estimate has the explicit form $\theta_1 = \frac{\varphi(t_0)-\varphi(t_2)}{\varphi(T)-\varphi(t_2){+\kappa}}$ {for a fixed $\kappa>0$}. This indicates that the ill-posedness of the problem becomes stronger as $t_0$ decreases.
\end{remark}

{
\begin{proof}[Proof of Theorem \ref{thm:riesz_stability}] 
By carrying out the same argument as in the proof of Theorem \ref{thm1}, with the definition replaced by the Riesz definition of the fractional Laplacian in Definition \ref{def:riesz_laplacian}, we immediately obtain Theorem \ref{thm:riesz_stability}, with the same exponent $\theta_1$ as in Theorem \ref{thm3}
\end{proof}
}

\section{Tikhonov Regularization}\label{sec5}
We consider the backward problem for the initial-boundary value problem \eqref{eq-main}: given the final-time measurements \(d_u(x) = u(x,T)\) and \(d_v(x) = \partial_t u(x,T)\) over the domain \(\Omega\), reconstruct the state \(\phi(x) = (u(x,t_0), \partial_t u(x,t_0))^T\) at any past time \(0 \leq t_0 < T\). This problem is severely ill-posed, as small perturbations in the final data can lead to large errors in the reconstructed past state. Based on the established conditional stability estimate, we propose a Tikhonov regularization-based minimization approach and derive the gradient via the variational adjoint method for efficient numerical solution. To simplify the derivation of the adjoint system, we henceforth assume that the coefficient $a$ is a constant in this section.

\subsection{Objective functional}
For the backward problem, we define the Tikhonov cost functional \(\mathcal{J}: (L^2(\Omega))^2 \to \mathbb{R}\) with an \(L^2\)-penalty term:
\begin{equation}
\mathcal{J}(\phi) = \frac{1}{2}\left\| u[\phi](\cdot,T) - d_u \right\|_{L^2(\Omega)}^2 + \frac{1}{2}\left\| \partial_t u[\phi](\cdot,T) - d_v \right\|_{L^2(\Omega)}^2 + \frac{\alpha}{2}\left\| \phi - \phi^b \right\|_{(L^2(\Omega))^2}^2,
\label{eq:cost_functional}
\end{equation}
where \(u[\phi]\) denotes the solution to the forward problem \eqref{eq-main} with initial conditions \(u(x,t_0) = b_0(x)\) and \(\partial_t u(x,t_0) = b_1(x)\) (i.e., \(\phi = (b_0, b_1)^T\)), \(\phi^b = (b_0^b, b_1^b)^T\) is the background prior state, and \(\alpha > 0\) is the regularization parameter that balances the data fidelity term and the regularization term. The minimizer \(\phi^* = \arg\min_{\phi} \mathcal{J}(\phi)\) is taken as the regularized solution to the backward problem. To compute the minimizer efficiently, we use the variational adjoint method to derive the gradient \(\nabla \mathcal{J}(\phi)\).
\subsubsection{Gateaux derivative and sensitivity system}
{In the following derivation, $\mathcal A_s=(-\Delta)^s$ is the spectral Dirichlet realization in Definition~\ref{def:spectral_laplacian}, and $a\in\mathbb R$ is constant. The homogeneous Dirichlet conditions below are understood through the operator domains specified in Section~\ref{sec2}.}
For arbitrary perturbation \(\hat{\phi} = (\hat{u}_0, \hat{v}_0)^T \in (L^2(\Omega))^2\) and small parameter \(\varepsilon > 0\), define the perturbed solution \(u^\varepsilon = u[\phi + \varepsilon \hat{\phi}]\) to the forward problem \eqref{eq-main} with initial conditions \(u^{{\varepsilon}}(x,t_0) = b_0 + \varepsilon \hat{u}_0\) and \(\partial_t u^{{\varepsilon}}(x,t_0) = b_1 + \varepsilon \hat{v}_0\). The first-order Gateaux derivative of the cost functional \(\mathcal{J}(\phi)\) is given by:
\begin{equation}
\begin{aligned}
\mathcal{J}'(\phi; \hat{\phi}) &= \lim_{\varepsilon \to 0} \frac{\mathcal{J}(\phi + \varepsilon \hat{\phi}) - \mathcal{J}(\phi)}{\varepsilon} \\
&= \int_\Omega \left( u[\phi](x,T) - d_u(x) \right) \lim_{\varepsilon \to 0} \frac{u^\varepsilon(x,T) - u[\phi](x,T)}{\varepsilon} dx \\
&\quad + \int_\Omega \left( \partial_t u[\phi](x,T) - d_v(x) \right) \lim_{\varepsilon \to 0} \frac{\partial_t u^\varepsilon(x,T) - \partial_t u[\phi](x,T)}{\varepsilon} dx \\
&\quad + \alpha \int_\Omega \left( (b_0(x) - b_0^b(x)) \hat{u}_0(x) + (b_1(x) - b_1^b(x)) \hat{v}_0(x) \right) dx.
\end{aligned}
\label{eq:gateaux_def}
\end{equation}
Define the sensitivity variables:
\[
\hat{u}(x,t) = \lim_{\varepsilon \to 0} \frac{u^\varepsilon(x,t) - u[\phi](x,t)}{\varepsilon}, \quad \hat{v}(x,t) = \partial_t \hat{u}(x,t).
\]
Substituting \(u^\varepsilon = u + \varepsilon \hat{u} + o(\varepsilon)\) into the forward problem \eqref{eq-main} and collecting terms of order \(\varepsilon\), we obtain the linearized sensitivity system:
\begin{equation}
\begin{cases}
\partial_t^2 \hat{u} + (-\Delta)^\delta \hat{u} + \partial_t (-\Delta)^\beta \hat{u} + a \partial_t (-\Delta)^{\gamma/2} \hat{u} = 0, & (x,t) \in \Omega \times (t_0, T), \\
\hat{u}(x, t_0) = \hat{u}_0(x), & x \in \Omega, \\
\partial_t \hat{u}(x, t_0) = \hat{v}_0(x), & x \in \Omega, \\
\hat{u}(x,t) = 0, & (x,t) \in \partial\Omega \times (t_0, T).
\end{cases}
\label{eq:sensitivity_system}
\end{equation}
Using the sensitivity variables, the Gateaux derivative can be rewritten as:
\begin{equation}
\mathcal{J}'(\phi; \hat{\phi}) = \int_\Omega \left( (u(T)-d_u) \hat{u}(T) + (\partial_t u(T)-d_v) \hat{v}(T) \right) dx + \alpha \langle \phi - \phi^b, \hat{\phi} \rangle_{(L^2(\Omega))^2}.
\label{eq:gateaux_final}
\end{equation}

\subsubsection{Adjoint problem and gradient derivation}
{For $L^2(\Omega)$ terminal data, the adjoint is understood in the transposition sense, with prescribed terminal traces $u^*(T)$ and $[-\partial_tu^*+(-\Delta)^\beta u^*+a(-\Delta)^{\gamma/2}u^*](T)$ in $L^2(\Omega)$.}
We construct the adjoint variable \(u^*(x,t)\) and multiply both sides of the sensitivity equation \eqref{eq:sensitivity_system} by \(u^*\), then integrate over the entire space-time domain \(\Omega \times (t_0, T)\):
\begin{equation}
\int_{t_0}^T \int_\Omega u^* \left( \partial_t^2 \hat{u} + (-\Delta)^\delta \hat{u} + \partial_t (-\Delta)^\beta \hat{u} + a \partial_t (-\Delta)^{\gamma/2} \hat{u} \right) dxdt = 0.
\label{eq:adjoint_integral}
\end{equation}
We now integrate each term by parts separately.

For the second-order time derivative term, we have:
\[
\int_{t_0}^T \int_\Omega u^* \partial_t^2 \hat{u} dxdt = \int_\Omega \left[ u^* \partial_t \hat{u} - \hat{u} \partial_t u^* \right]_{t_0}^T dx + \int_{t_0}^T \int_\Omega \hat{u} \partial_t^2 u^* dxdt.
\]
For the fractional diffusion term, the self-adjointness directly gives:
\[
\int_{t_0}^T \int_\Omega u^* (-\Delta)^\delta \hat{u} dxdt = \int_{t_0}^T \int_\Omega \hat{u} (-\Delta)^\delta u^* dxdt.
\]
For the fractional damping term, we first integrate by parts in time and then use the self-adjointness of the fractional Laplacian:
\[
\begin{aligned}
\int_{t_0}^T \int_\Omega u^* \partial_t (-\Delta)^\beta \hat{u} dxdt &= \int_\Omega \left[ u^* (-\Delta)^\beta \hat{u} \right]_{t_0}^T dx - \int_{t_0}^T \int_\Omega (-\Delta)^\beta \hat{u} \partial_t u^* dxdt \\
&= \int_\Omega \left[ u^* (-\Delta)^\beta \hat{u} \right]_{t_0}^T dx - \int_{t_0}^T \int_\Omega \hat{u} \partial_t (-\Delta)^\beta u^* dxdt.
\end{aligned}
\]
For the constant coefficient fractional damping term, since \(a\) is constant and commutes with spatial operators, we follow a similar procedure:
\[
\begin{aligned}
\int_{t_0}^T \int_\Omega a u^* \partial_t (-\Delta)^{\gamma/2} \hat{u} dxdt &= a \int_\Omega \left[ u^* (-\Delta)^{\gamma/2} \hat{u} \right]_{t_0}^T dx - a \int_{t_0}^T \int_\Omega (-\Delta)^{\gamma/2} \hat{u} \partial_t u^* dxdt \\
&= a \int_\Omega \left[ u^* (-\Delta)^{\gamma/2} \hat{u} \right]_{t_0}^T dx - a \int_{t_0}^T \int_\Omega \hat{u} \partial_t (-\Delta)^{\gamma/2} u^* dxdt.
\end{aligned}
\]

Substituting all integrated terms back into \eqref{eq:adjoint_integral} and rearranging, we get:
\begin{equation}
\begin{aligned}
0 &= \int_{t_0}^T \int_\Omega \hat{u} \left( \partial_t^2 u^* + (-\Delta)^\delta u^* - \partial_t (-\Delta)^\beta u^* - a \partial_t (-\Delta)^{\gamma/2} u^* \right) dxdt \\
&\quad + \int_\Omega \left[ u^* \partial_t \hat{u} - \hat{u} \partial_t u^* + u^* (-\Delta)^\beta \hat{u} + a u^* (-\Delta)^{\gamma/2} \hat{u} \right]_{t_0}^T dx.
\end{aligned}
\label{eq:integrated_terms}
\end{equation}
To eliminate the volume integral term involving \(\hat{u}\), we require the adjoint variable \(u^*\) to satisfy the following backward-in-time system:
\begin{equation}
\partial_t^2 u^* + (-\Delta)^\delta u^* - \partial_t (-\Delta)^\beta u^* - a \partial_t (-\Delta)^{\gamma/2} u^* = 0, \quad (x,t) \in \Omega \times (t_0, T),
\label{eq:adjoint_equation}
\end{equation}
with homogeneous Dirichlet boundary condition \(u^*(x,t) = 0\) on \(\partial\Omega \times (t_0, T)\).

With this, equation \eqref{eq:integrated_terms} reduces to the boundary term identity:
\begin{equation}
\begin{aligned}
    &\left. \int_\Omega \left( u^* \partial_t \hat{u} - \hat{u} \partial_t u^* + u^* (-\Delta)^\beta \hat{u} + a u^* (-\Delta)^{\gamma/2} \hat{u} \right) dx \right|_{t=T} \\
    =& \left. \int_\Omega \left( u^* \partial_t \hat{u} - \hat{u} \partial_t u^* + u^* (-\Delta)^\beta \hat{u} + a u^* (-\Delta)^{\gamma/2} \hat{u} \right) dx \right|_{t=t_0}.
\end{aligned}
\end{equation}

Rearranging the terminal terms at \(t=T\) and using the initial conditions of the sensitivity system \(\hat{u}(t_0) = \hat{u}_0\) and \(\partial_t \hat{u}(t_0) = \hat{v}_0\), we obtain:
\begin{equation}
\begin{aligned}
&\int_\Omega \left( u^*(T) \partial_t \hat{u}(T) + \hat{u}(T) \left( -\partial_t u^*(T) + (-\Delta)^\beta u^*(T) + a (-\Delta)^{\gamma/2} u^*(T) \right) \right) dx \\
=& \int_\Omega \left( u^*(t_0) \hat{v}_0 + \hat{u}_0 \left( -\partial_t u^*(t_0) + (-\Delta)^\beta u^*(t_0) + a (-\Delta)^{\gamma/2} u^*(t_0) \right) \right) dx.
\end{aligned}
\label{eq:boundary_terms}
\end{equation}

We now choose the terminal conditions for the adjoint system to exactly match the data fidelity terms in the Gateaux derivative \eqref{eq:gateaux_final}. Comparing the left-hand side of \eqref{eq:boundary_terms} with the data terms in \eqref{eq:gateaux_final}, we set:
\[
\begin{cases}
u^*(x, T) = \partial_t u(x,T) - d_v(x), \\
-\partial_t u^*(x, T) + (-\Delta)^\beta u^*(x,T) + a (-\Delta)^{\gamma/2} u^*(x,T) = u(x,T) - d_u(x).
\end{cases}
\]
{Together with the homogeneous Dirichlet condition, these terminal conditions give the complete adjoint system}
\begin{equation}
\begin{cases}
\partial_t^2 u^* + (-\Delta)^\delta u^* - \partial_t (-\Delta)^\beta u^* - a \partial_t (-\Delta)^{\gamma/2} u^* = 0 & \text{in } \Omega \times (t_0, T), \\
u^*(\cdot, T) = \partial_t u(\cdot,T) - d_v & \text{in }\Omega, \\
{-\partial_t u^*(\cdot,T)+(-\Delta)^\beta u^*(\cdot,T)+a(-\Delta)^{\gamma/2}u^*(\cdot,T)=u(\cdot,T)-d_u} & \text{in }\Omega, \\
u^* = 0 &  \text{on }\partial\Omega \times (t_0, T).
\end{cases}
\label{eq:adjoint_system_full}
\end{equation}

Substituting \eqref{eq:boundary_terms} into the Gateaux derivative \eqref{eq:gateaux_final}, we eliminate the unknown terminal terms and obtain:
\[
\begin{aligned}
\mathcal{J}'(\phi; \hat{\phi}) &= \int_\Omega \left( -\partial_t u^*(t_0) + (-\Delta)^\beta u^*(t_0) + a (-\Delta)^{\gamma/2} u^*(t_0) \right) \hat{u}_0 dx \\
&\quad + \int_\Omega u^*(t_0) \hat{v}_0 dx + \alpha \int_\Omega \left( (b_0 - b_0^b) \hat{u}_0 + (b_1 - b_1^b) \hat{v}_0 \right) dx.
\end{aligned}
\]
By the definition of the \(L^2\) gradient, this implies:
\begin{equation}
\nabla \mathcal{J}(\phi) = \begin{pmatrix}
-\partial_t u^*(\cdot,t_0) + (-\Delta)^\beta u^*(\cdot,t_0) + a (-\Delta)^{\gamma/2} u^*(\cdot,t_0) + \alpha (b_0 - b_0^b) \\
u^*(\cdot,t_0) + \alpha (b_1 - b_1^b)
\end{pmatrix}.
\label{eq:gradient}
\end{equation}

\begin{remark}
{
For the zero-exterior Riesz realization in Definition~\ref{def:riesz_laplacian}, the same calculation uses $\mathcal A_s=\mathcal L_s$ for every spatial order and the homogeneous exterior conditions
\[
\hat u(x,t)=u^*(x,t)=0,
\qquad (x,t)\in(\mathbb R^n\setminus\Omega)\times(t_0,T).
\]
Indeed, self-adjointness gives
\(
(\mathcal L_s v,w)_{L^2(\Omega)}
=(v,\mathcal L_s w)_{L^2(\Omega)}\), \(v,w\in D(\mathcal L_s)\). Consequently, the sensitivity equation, adjoint equation, terminal conditions, and gradient formula above are obtained with $(-\Delta)^s$ replaced by $\mathcal L_s$ and the Dirichlet boundary conditions replaced by these exterior conditions.
}
\end{remark}

\subsection{Convergence rate analysis of Tikhonov regularization}
In this section, we establish the convergence rate of Tikhonov regularization based on the conditional stability result introduced in Theorem \ref{thm3}. {We use the spectral Dirichlet realization with $0<\beta\leq1$, $0\leq\delta\leq\beta$, and $a=0$, and impose the stronger prior below.} {
For the convergence analysis in this subsection, we set the background state
$\phi^b=0$.
}

{
\begin{lemma}\label{lem:strongprior}
Let $\epsilon>0$, $M>0$, and suppose that $u$ solves \eqref{eq-main} with spectral definition and satisfies the stronger a priori bound
\begin{equation}\label{eq:strong_history_prior}
\|u\|_{H^1(0,T;D(( -\Delta)^{\beta+\epsilon}))}\leq M.
\end{equation}
Then there exists $C_\epsilon>0$, independent of $u$ and $M$, such that
\[
\begin{aligned}
\|u\|_{H^1(0,T;D(( -\Delta)^\beta))}
+\|\partial_t^2u\|_{L^2(0,T;D(( -\Delta)^\epsilon))}
\leq C_\epsilon M,
\end{aligned}
\]
and
\[
\begin{aligned}
\|u\|_{C([0,T];D(( -\Delta)^{\beta+\epsilon}))}+\|\partial_tu\|_{C([0,T];D(( -\Delta)^{\beta/2+\epsilon}))}
\leq C_\epsilon M.
\end{aligned}
\]
\end{lemma}

\begin{proof}
Since $\delta\leq\beta$, spectral-domain embeddings and the equation give
\[
\partial_t^2u=-(-\Delta)^\delta u-(-\Delta)^\beta\partial_tu,
\qquad
\|\partial_t^2u\|_{L^2(0,T;D(( -\Delta)^\epsilon))}
\leq C_\epsilon M.
\]
Consequently,
\[
\partial_tu\in H^1(0,T;D(( -\Delta)^\epsilon))
\cap L^2(0,T;D(( -\Delta)^{\beta+\epsilon})).
\]
The time trace theorem gives
\[
\partial_tu\in C([0,T];D(( -\Delta)^{\beta/2+\epsilon})),
\qquad
\|\partial_tu\|_{C([0,T];D(( -\Delta)^{\beta/2+\epsilon}))}
\leq C_\epsilon M.
\]
The bound for $u$ follows directly from \eqref{eq:strong_history_prior}.
\end{proof}
}

Now we define the spaces
\[
X = {D(( -\Delta)^\beta)}\times H^\beta_0(\Omega),\qquad
X_0 = L^2(\Omega)\times L^2(\Omega),
\]
endowed with the norms
\[
\|v_0\|_X^2 = \|v_1\|_{{D(( -\Delta)^\beta)}}^2 + \|v_2\|_{{H_0^\beta(\Omega)}}^2,\qquad
\|(w_1,w_2)\|_{X_0}^2 = \|w_1\|_{L^2(\Omega)}^2 + \|w_2\|_{L^2(\Omega)}^2.
\]
In this section we assume that the observations are given in the weaker space $X_0$, equipped with its natural norm. The forward operator $\mathcal{F}_{t_0}:X\to Z$ maps the intermediate state at time $t_0$ to the final observations at time $T$, where $Z = {D(( -\Delta)^\beta)\times H_0^\beta(\Omega)=X}$ is the natural space for the data if the initial state is in $X$. However, the noise is measured in the $X_0$ norm. The forward operator is linear {and} bounded; in particular there exists $C_{\mathrm{reg}}>0$ such that
\[
\|\mathcal{F}_{t_0}v_0\|_Z \le C_{\mathrm{reg}}\|v_0\|_X\qquad\forall {v_0}\in X.
\]

{
The mode equations $\partial_t^2u_k+\lambda_k^\beta\partial_tu_k+\lambda_k^\delta u_k=0$ give a bounded extension $\mathcal F_{t_0}:X_0\to X_0$. Set $Z_\epsilon=D(( -\Delta)^{\beta+\epsilon})\times D(( -\Delta)^{(\beta+\epsilon)/2})$. Lemma~\ref{lem:strongprior} implies
\begin{equation}\label{eq:terminal_prior}
\|(u(T),\partial_tu(T))\|_{Z_\epsilon}\leq C_\epsilon M,
\end{equation}
by $D(( -\Delta)^{\beta/2+\epsilon})\hookrightarrow D(( -\Delta)^{(\beta+\epsilon)/2})$. Impose \eqref{eq:strong_history_prior} on every admissible history on $[0,T]$ and define
\begin{equation}\label{eq:admissible_states}
\begin{aligned}
\mathcal K_{M,\epsilon}(t_0)=\bigl\{(u(t_0),\partial_tu(t_0)),\|u\|_{H^1(0,T;D(( -\Delta)^{\beta+\epsilon}))}\leq M
\bigr\}.
\end{aligned}
\end{equation}

The set $\mathcal K_{M,\epsilon}(t_0)$ is nonempty and convex. Its histories are bounded in $H^1(0,T;D(( -\Delta)^{\beta+\epsilon}))
\cap H^2(0,T;D(( -\Delta)^\epsilon))$. Weak limits preserve the equation, initial condition, and prior. The traces are bounded in $D(( -\Delta)^{\beta+\epsilon})
\times D(( -\Delta)^{\beta/2+\epsilon})$, which embeds compactly into $X_0$}.
\begin{lemma}
\label{lem:stabilityX0}
{Fix $\epsilon>0$ and $t_0\in(0,T)$.} There exist constants $C_0>0$ and $\theta\in(0,1)$ such that for any solution $u$ of \eqref{eq-main} satisfying
\[
\|u\|_{H^1(0,T;{D(( -\Delta)^{\beta+\epsilon})})}\le M,
\]
the intermediate state $v_0 = (u(\cdot,t_0),\partial_t u(\cdot,t_0))$ fulfills
\[
\|v_0\|_{X_0} \le {C_0} M^{1-\theta}\,\|\mathcal{F}_{t_0}v_0\|_{X_0}^{\theta}.
\]
{Here $\theta=\theta_1\epsilon/(\beta+\epsilon)$ with $\theta_1$ as in Remark~\ref{remark4.1}, and $C_0$ is independent of $u$ and $M$.}
\end{lemma}

\begin{proof}
{By \eqref{eq:terminal_prior},}
\[
\|\mathcal F_{t_0}v_0\|_{Z_\epsilon}
\leq C_\epsilon{M}.
\]

We now use interpolation between the weakest space $X_0$ and the higher regularity space $Z_\epsilon$. For the first component,
\[
\|(\mathcal{F}_{t_0}v_0)_{{1}}\|_{{D(( -\Delta)^\beta)}} \le C_{\mathrm{int}} \|(\mathcal{F}_{t_0}v_0)_{{1}}\|_{L^2}^{1-\gamma} \|(\mathcal{F}_{t_0}v_0)_{{1}}\|_{{D(( -\Delta)^{\beta+\epsilon})}}^{\gamma},
\]
with $\gamma = \frac{2\beta}{2\beta+2\epsilon} = \frac{\beta}{\beta+\epsilon}$. The same interpolation exponent holds for the second component ${H_0^\beta(\Omega)=D(( -\Delta)^{\beta/2})}$ between $L^2$ and ${D(( -\Delta)^{(\beta+\epsilon)/2})}$. Therefore,
\begin{equation}\label{interpolation}
    \|\mathcal{F}_{t_0}v_0\|_Z \le C_{\mathrm{int}} \|\mathcal{F}_{t_0}v_0\|_{X_0}^{1-\gamma} \|\mathcal{F}_{t_0}v_0\|_{Z_\epsilon}^{\gamma}
\le C_{\mathrm{int}} C_\epsilon^{\gamma} \|\mathcal{F}_{t_0}v_0\|_{X_0}^{1-\gamma} {M^{\gamma}}. 
\end{equation}

According to Theorem \ref{thm3} {and Lemma~\ref{lem:strongprior}}, we have
\[
\|v_0\|_{X_0} \le C_1 M^{1-\theta_1} \|\mathcal{F}_{t_0}v_0\|_Z^{\theta_1}.
\]
Inserting the interpolation bound \eqref{interpolation} yields
\[
\|v_0\|_{X_0} \le C_1 M^{1-\theta_1} \bigl(C_{\mathrm{int}} C_\epsilon^{\gamma} \|\mathcal{F}_{t_0}v_0\|_{X_0}^{1-\gamma} {M^{\gamma}}\bigr)^{\theta_1}
= C_2 M^{1-\theta_1} \|\mathcal{F}_{t_0}v_0\|_{X_0}^{\theta_1(1-\gamma)} {M^{\theta_1\gamma}}.
\]

{Consequently,}
\[
\|v_0\|_{X_0} \le C_3 M^{1-\theta_1+\theta_1\gamma} \|\mathcal{F}_{t_0}v_0\|_{X_0}^{\theta_1(1-\gamma)}.
\]
Setting
\[
\theta = \theta_1(1-\gamma) = \theta_1\frac{\epsilon}{\beta+\epsilon}\in(0,1),
\]
we obtain
\[
\|v_0\|_{X_0} \le C M^{1-\theta} \|\mathcal{F}_{t_0}v_0\|_{X_0}^{\theta}.
\]
This completes the proof of the lemma.
\end{proof}

Now let $v_0^\dagger\in X_0$ be the true intermediate state, corresponding to a solution $u^\dagger$ of \eqref{eq-main}.  We assume that $u^\dagger$ possesses the {stronger a priori bound}
\[
\|u^\dagger\|_{H^1(0,T;{D(( -\Delta)^{\beta+\epsilon})})} \le M.
\]
Let $y^\sigma\in X_0$ be noisy observations satisfying
\[
\|y^\sigma - \mathcal{F}_{t_0}v_0^\dagger\|_{X_0} \le \sigma,\qquad \sigma>0.
\]
We consider the Tikhonov functional
\[
J_\alpha({v_0}) = \|\mathcal{F}_{t_0}v_0 - y^\sigma\|_{X_0}^2 + \alpha\|v_0\|_{X_0}^2,\qquad \alpha>0,\quad {v_0\in\mathcal K_{M,\epsilon}(t_0)},
\]
and let $v_0^{\alpha,\sigma}\in {\mathcal K_{M,\epsilon}(t_0)}$ denote its unique minimizer {over this admissible set}. {Compactness and strict convexity give existence and uniqueness. The true state $v_0^\dagger$ is admissible.} Standard properties of Tikhonov regularization yield the following estimates.

\begin{lemma}
\label{lem:apriori}
Let $M_0 = \|v_0^\dagger\|_{X_0}$. Then
\[
\|\mathcal{F}_{t_0}v_0^{\alpha,\sigma} - y^\sigma\|_{X_0} \le \sigma + \sqrt{\alpha}M_0,\qquad
\|v_0^{\alpha,\sigma}\|_{X_0} \le \sqrt{\frac{\sigma^2}{\alpha}+M_0^2}.
\]
Consequently,
\begin{equation}\label{5.13}
    \|\mathcal{F}_{t_0}v_0^{\alpha,\sigma} - \mathcal{F}_{t_0}v_0^\dagger\|_{X_0} \le 2\sigma + \sqrt{\alpha}M_0.
\end{equation}
\end{lemma}

\begin{proof}
Since $v_0^{\alpha,\sigma}$ minimizes $J_\alpha(v)=\|\mathcal{F}_{t_0}v-y^\sigma\|_{X_0}^2+\alpha\|v\|_{X_0}^2$ {over $\mathcal K_{M,\epsilon}(t_0)$ and $v_0^\dagger$ belongs to this set}, we have $J_\alpha(v_0^{\alpha,\sigma})\le J_\alpha(v_0^\dagger)$. Using $\|y^\sigma-\mathcal{F}_{t_0}v_0^\dagger\|_{X_0}\le\sigma$ and $\|v_0^\dagger\|_{X_0}=M_0$, we obtain
\[
\|\mathcal{F}_{t_0}v_0^{\alpha,\sigma} - y^\sigma\|_{X_0}^2 + \alpha\|v_0^{\alpha,\sigma}\|_{X_0}^2
\le \sigma^2 + \alpha M_0^2.
\]
The first two inequalities follow by dropping one nonnegative term on the left and using $\sqrt{a^2+b^2}\le a+b$ for $a,b\ge0$. For \eqref{5.13}, apply the triangle inequality:
\[
\|\mathcal{F}_{t_0}v_0^{\alpha,\sigma} - \mathcal{F}_{t_0}v_0^\dagger\|_{X_0}
\le \|\mathcal{F}_{t_0}v_0^{\alpha,\sigma} - y^\sigma\|_{X_0} + \|y^\sigma - \mathcal{F}_{t_0}v_0^\dagger\|_{X_0}
\le (\sigma+\sqrt{\alpha}M_0) + \sigma = 2\sigma+\sqrt{\alpha}M_0.
\]
We complete the proof of the lemma.
\end{proof}

\begin{proof}[\bf Proof of Theorem \ref{thm:main}]
Let $u^{\alpha,\sigma}$ be {an admissible history for $v_0^{\alpha,\sigma}$ on $[0,T]$}. {Then}
\[
\|u^{\alpha,\sigma}\|_{{H^1(0,T;D(( -\Delta)^{\beta+\epsilon}))}}
{\leq M}.
\]
{Lemma~\ref{lem:strongprior} gives $M_0\leq C_{\mathrm{trace}}M$.} The difference $w = u^{\alpha,\sigma} - u^\dagger$ solves the homogeneous equation {on $[0,T]$, with state $v_0^{\alpha,\sigma}-v_0^\dagger$ at $t_0$}. By linearity,
\[
\|w\|_{H^1({0},T;{D(( -\Delta)^{\beta+\epsilon})})}
\le \|u^{\alpha,\sigma}\|_{H^1({0},T;{D(( -\Delta)^{\beta+\epsilon})})} + \|u^\dagger\|_{H^1({0},T;{D(( -\Delta)^{\beta+\epsilon})})}
{\le 2M =: C_2.}
\]
Applying Lemma~\ref{lem:stabilityX0} to $w$ {with the prior radius $C_2=2M$} yields

\begin{equation}\label{residualestimate}
    \|v_0^{\alpha,\sigma} - v_0^\dagger\|_{X_0}
\le C_0 C_2^{1-\theta} \|\mathcal{F}_{t_0}(v_0^{\alpha,\sigma}-v_0^\dagger)\|_{X_0}^{\theta}. 
\end{equation}

From the residual estimate {\eqref{5.13}} and $\alpha=c\sigma^2$ we have
\[
\|\mathcal{F}_{t_0}(v_0^{\alpha,\sigma}-v_0^\dagger)\|_{X_0} \le 2\sigma + \sqrt{c}M_0\,\sigma \le (2 + \sqrt{c}\,C_{\mathrm{trace}} M)\sigma.
\]
Substituting this into {\eqref{residualestimate}} gives
\[
\|v_0^{\alpha,\sigma} - v_0^\dagger\|_{X_0}
\le C_0 C_2^{1-\theta} \bigl(2 + \sqrt{c}\,C_{\mathrm{trace}} M\bigr)^{\theta} \sigma^{\theta}.
\]
Choosing $C = C_0 C_2^{1-\theta} (2 + \sqrt{c}\,C_{\mathrm{trace}} M)^{\theta}$ completes the proof of the theorem.
\end{proof}
\begin{remark}
The exponent $\theta =  \frac{\epsilon \theta_1}{\beta+\epsilon}$ {expresses the effect of the higher-order a priori assumption on conditional stability.} {Larger} $\epsilon$ makes $\theta$ closer to $\theta_1$, yielding a faster convergence rate $\sigma^{\theta}$. The constant $C_0$ depends on $\epsilon$ through $C_\epsilon$ and may grow as $\epsilon$ increases, in practice $\epsilon$ is fixed according to the available a priori regularity.
\end{remark}

{
\begin{remark}
\label{rem:riesz_tikhonov}
Under the assumptions of Theorem~\ref{thm:riesz_stability} with $a=0$, the constrained Tikhonov argument also applies to the Riesz realization. Set $\eta=\epsilon/\beta>0$ and use $X=Z=D(\mathcal L_\beta)\times\mathcal H_0^\beta(\Omega)$. Assume that the true solution and every admissible history satisfy
\begin{equation}\label{eq:riesz_strong_prior}
\|u\|_{H^1(0,T;D(\mathcal L_\beta^{1+\eta}))}
+\|u\|_{H^2(0,T;D(\mathcal L_\beta^\eta))}\leq M.
\end{equation}
Here $\mathcal L_\beta^r$ denotes the self-adjoint functional-calculus power of the fixed zero-exterior operator $\mathcal L_\beta$.
The time trace theorem bounds the terminal state in
\[
Z_\epsilon^{\mathrm R}
=D(\mathcal L_\beta^{1+\eta})
\times D(\mathcal L_\beta^{(1+\eta)/2}).
\]
Interpolation with $X_0$ and Theorem~\ref{thm:riesz_stability} give the rate $O(\sigma^\theta)$, where
\[
\theta=\theta_1\frac{\eta}{1+\eta}
=\theta_1\frac{\epsilon}{\beta+\epsilon}.
\]
The minimization is over the corresponding full-history admissible set defined by \eqref{eq:riesz_strong_prior}. For $\delta=\beta$, the equation controls the $H^2$ term by the $H^1$ norm in this prior.
\end{remark}
}

\section{Numerical Experiments}\label{sec6}

We consider $\Omega=(0,1)^2$ and reconstruct the state $\phi=(u(\cdot,t_0),\partial_tu(\cdot,t_0))^T$ at $0<t_0<T$ from final-time observations using Tikhonov regularization. Unless otherwise stated, we set $a=1$.

Let $u^\dagger$ denote the reference solution and $u^{\alpha,\sigma}$ its reconstruction. Throughout this section, we report the relative $L^2$ errors
\begin{equation}\label{eq:numerical_relative_errors}
\begin{aligned}
e_u&=\frac{\|u^{\alpha,\sigma}(\cdot,t_0)-u^\dagger(\cdot,t_0)\|_{L^2(\Omega)}}
{\|u^\dagger(\cdot,t_0)\|_{L^2(\Omega)}},\\
e_v&=\frac{\|\partial_tu^{\alpha,\sigma}(\cdot,t_0)-\partial_tu^\dagger(\cdot,t_0)\|_{L^2(\Omega)}}
{\|\partial_tu^\dagger(\cdot,t_0)\|_{L^2(\Omega)}}.
\end{aligned}
\end{equation}

The noisy data are simulated as
{
\[
\phi_v^\sigma = \phi_v(x) + \sigma \cdot \eta(x)\, \phi_v(x),\quad x\in\Omega,
\]}
where $\sigma\geq0$ is the nominal noise amplitude and $\eta(x)$ is uniformly distributed in $[-1,1]$.

Since it is difficult to construct an exact analytical solution for fractional-order partial differential equations, we employ the spectral method solution as the approximate exact solution in our experiments. The spectral method offers higher efficiency for fractional-order PDEs, whereas finite element methods are less efficient and more challenging for implementing fractional derivatives. {We only use the spectral method to solve the \eqref{eq-main} under the spectral definition, see definition \ref{def:spectral_laplacian}; the finite element implementation for the Riesz case will be left for future consideration.}

{Example~\ref{exp1} provides supplementary reconstruction tests with a fixed regularization parameter, and Example~\ref{exp:alpha_sigma2} examines the parameter rule in Theorem~\ref{thm:main}. Unless otherwise stated, we set $a=1$.}
\begin{example}\label{exp1}
    The fractional orders are chosen as $\gamma=\beta=\delta=0.8$. We initialize the numerical tests with the following initial condition for the displacement:
    $$
    u_0(x,y)=\sin(\pi x)\sin(\pi y).
    $$
\end{example}

To verify the robustness of the proposed regularization method, we perform numerical tests under four different noise levels: $\sigma=0$, $0.001$, $0.01$ and $0.1$, respectively. The regularization parameter is fixed at $\alpha=10^{-5}$ for all test cases. {We use $50\times50$ interior nodes with $h=1/51$ and
Crank--Nicolson time stepping with $\Delta t=0.005$.
The L-BFGS-B solver is configured with gradient infinity-norm
and relative objective-decrease tolerances of $10^{-14}$ and
$10^{-20}$, respectively, and a maximum of $1000$ iterations.} The numerical results of the forward problem are illustrated in Figure~\ref{fig:true_solution_u_sin}. The inversion results under different noise levels are displayed in Figure~\ref{fig:reconstruction_u_sin}, and the errors $e_u$ and $e_v$ defined in \eqref{eq:numerical_relative_errors} are reported in Table~\ref{table1}.

\begin{figure}[ht]
    \centering
    \includegraphics[width=0.5\linewidth]{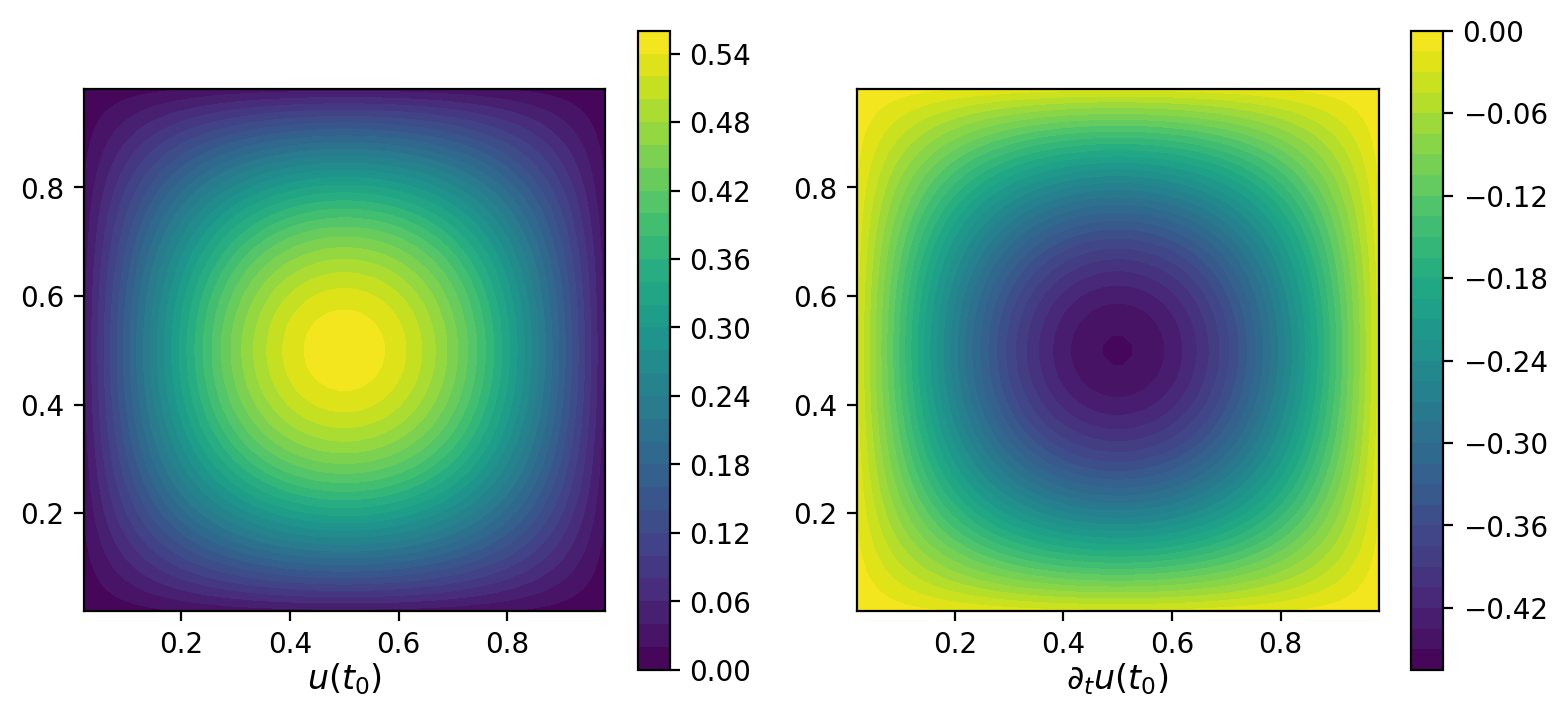}
    \caption{True $u(t_0)$ and $\partial_t u(t_0)$ when $t_0=0.8$}
    \label{fig:true_solution_u_sin}
\end{figure}

\begin{figure}[ht]
    \centering
    \includegraphics[width=1\linewidth]{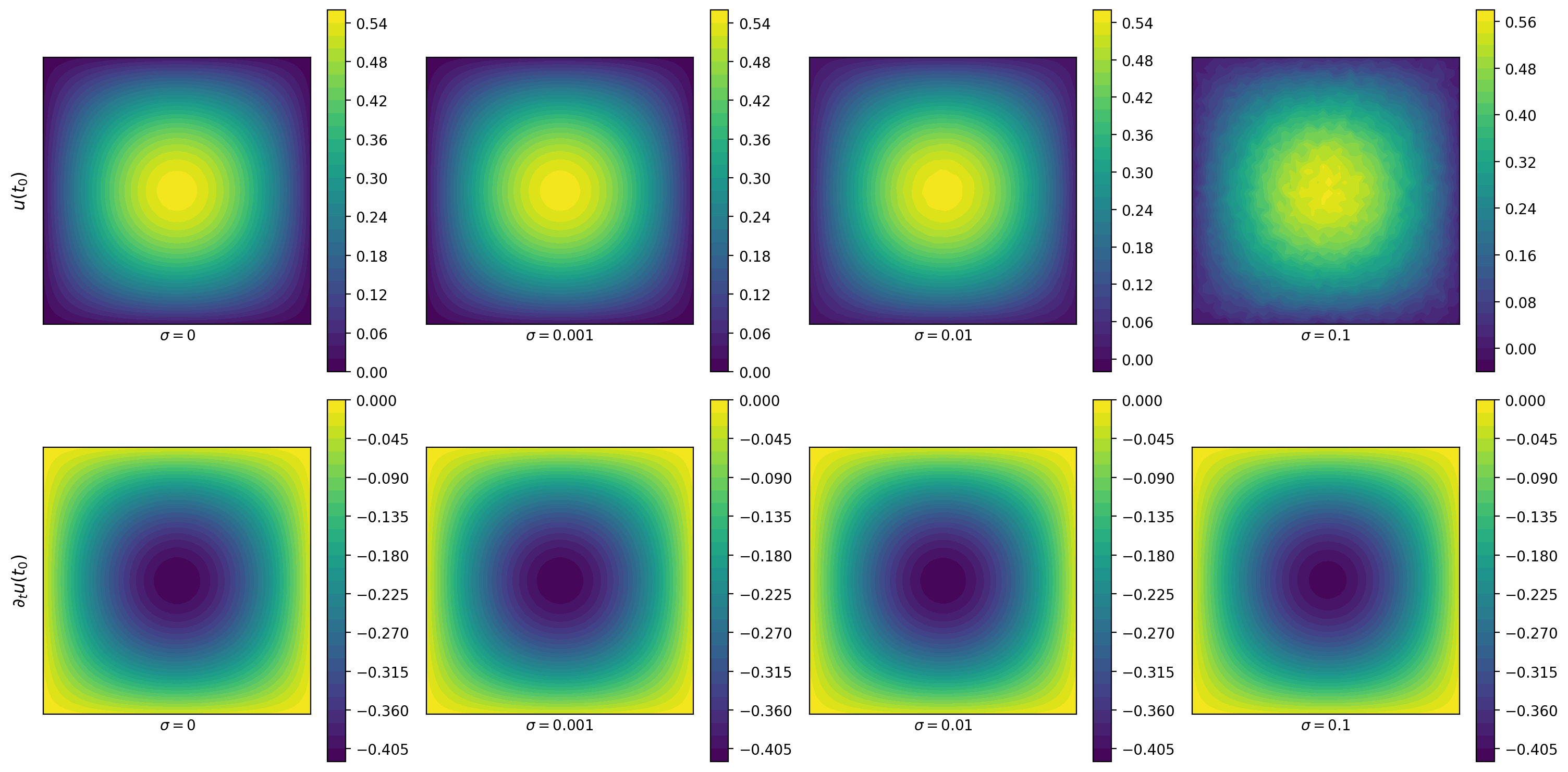}
    \caption{The reconstructions of $u(t_0)$ and $\partial_t u(t_0)$ under $\sigma=0,0.001,0.01,0.1$ when $t_0=0.8$}
    \label{fig:reconstruction_u_sin}
\end{figure}

\begin{table}[ht]
    \centering
    \caption{Relative errors $e_u$ and $e_v$ under different noise levels}
    \begin{tabular}{lcccc}
        \toprule
        $\sigma$ & \textbf{0} & \textbf{0.001} & \textbf{0.01} & \textbf{0.1} \\
        \midrule
        $e_u$ & $2.1262\times 10^{-3}$ & $2.1366\times 10^{-3}$ & $3.9080\times 10^{-3}$ & $3.3806\times 10^{-2}$ \\
        $e_v$ & $7.0512\times 10^{-2}$ & $7.0628\times 10^{-2}$ & $7.1672\times 10^{-2}$ & $8.2700\times 10^{-2}$ \\
        \bottomrule
    \end{tabular}
    \label{table1}
\end{table}
{
For $\sigma=0$, the direct discrete Tikhonov minimizer reproduces
$(e_u,e_v)=(2.1262\times10^{-3},7.0512\times10^{-2})$. Halving $\Delta t$ to $0.0025$ gives
$(e_u,e_v)=(2.1216\times10^{-3},7.0409\times10^{-2})$.
These checks identify regularization bias as the dominant
contribution to the nonzero errors. The larger velocity bias in this configuration is consistent with the weaker sensitivity of the terminal observations to
velocity perturbations at $t_0$.
}

The results demonstrate the robustness of our proposed algorithm against noise. It can be observed that accurate reconstructions can be obtained even when the noise level is as high as $0.1$. However, we note that the reconstruction accuracy is highly dependent on the choice of regularization parameters, especially for the velocity component. In particular, the reconstruction accuracy deteriorates significantly when the regularization parameter is set to $10^{-4}$ or $10^{-3}$ under noise level $\sigma=0.01$, as detailed in Table \ref{table4}. This observation further confirms the severe ill-posedness of the backward problem. In the subsequent examples, we will test more high-frequency functions and investigate the effects of the fractional order and the initial time $t_0$ on the reconstruction errors.

\begin{table}[ht]
  \centering
  \caption{Relative errors $e_u$ and $e_v$ under different regularization parameters $\alpha$}
  \begin{tabular}{lccc}
    \toprule
     $\alpha$& \textbf{$10^{-6}$} & \textbf{$10^{-4}$} & \textbf{$10^{-3}$} \\
    \midrule
    $e_u$ & $3.78\times 10^{-3}$ & $1.39\times 10^{-2}$ & $4.06\times 10^{-2}$ \\
    $e_v$ & $8.40\times 10^{-2}$ & $4.38\times 10^{-1}$ & $9.15\times 10^{-1}$ \\
    \bottomrule
  \end{tabular}
  \label{table4}
\end{table}

{
The experiments with $a=1$ illustrate reconstruction performance
for the generalized model, and the parameter rule in
Theorem~\ref{thm:main} is tested in
Example~\ref{exp:alpha_sigma2}, where $a=0$ and
$0\leq\delta\leq\beta$.
}
\begin{example}
    In this example, we set $a=0$ and select the parameter pairs $(\delta,\beta)$ as $(0.2,0.2)$, $(0.4,0.4)$, $(0.6,0.6)$ and $(0.8,0.8)$ to test the impact of the fractional orders. We perform numerical experiments with the following initial displacement condition:
    \begin{equation}
        u_0(x,y)=\sin(2\pi x)\sin(\pi y).
    \end{equation}
    The noise level is fixed at $\sigma=0.01$, and we choose $T=0.5$ and $t_0=0.4$. All other settings are identical to those in Example~\ref{exp1}. The exact solution is presented in Figure~\ref{fig:true_solution2}, and the corresponding reconstructed results are shown in Figure~\ref{fig:reconstruction_delta_beta}. The relative errors $e_u$ and $e_v$ are summarized in Table~\ref{table2}.
\end{example}

\begin{figure}[ht]
    \centering
    \includegraphics[width=0.5\linewidth]{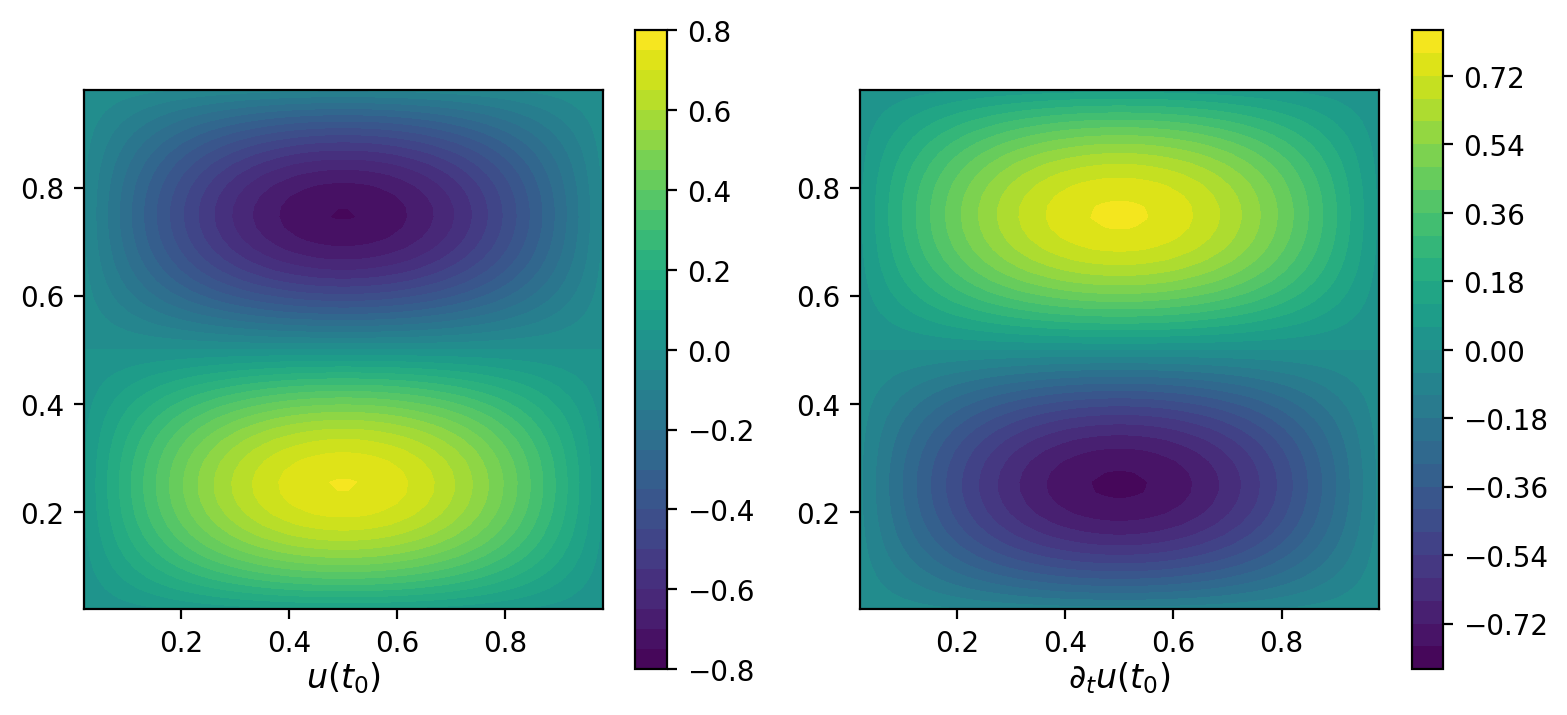}
    \caption{True $u(t_0)$ and $\partial_t u(t_0)$ when $\delta,\beta=0.4$}
    \label{fig:true_solution2}
\end{figure}

\begin{figure}[ht]
    \centering
    \includegraphics[width=1\linewidth]{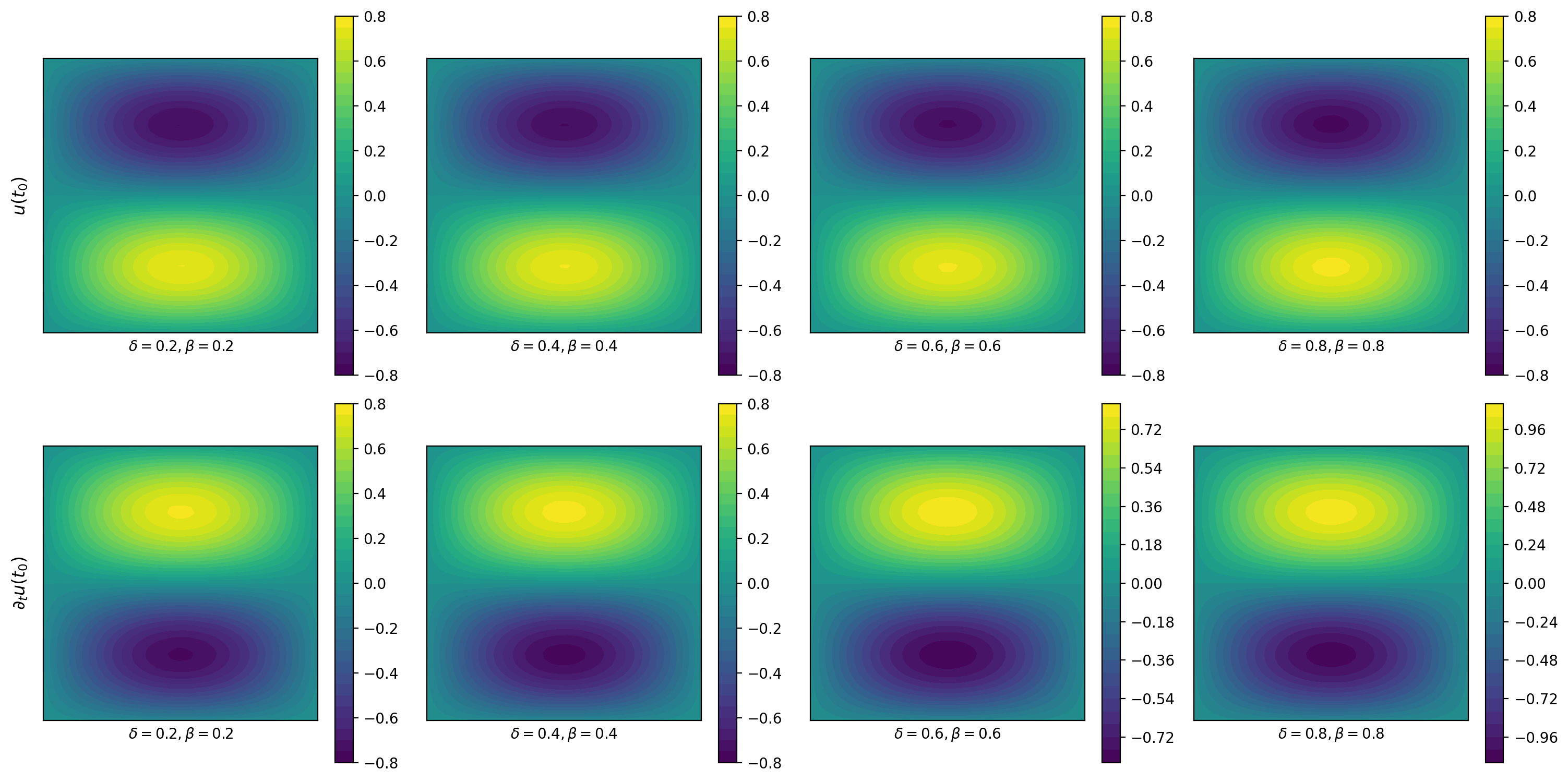}
    \caption{The reconstructions of $u(t_0)$ and $\partial_t u(t_0)$ with different $\beta$ and $\delta$ when $t_0=0.4$}
    \label{fig:reconstruction_delta_beta}
\end{figure}

\begin{table}[ht]
    \centering
    \caption{Relative errors $e_u$ and $e_v$ for different $(\delta,\beta)$}
    \begin{tabular}{lcccc}
        \toprule
        $(\delta,\beta)$ & \textbf{$(0.2, 0.2)$} & \textbf{$(0.4, 0.4)$} & \textbf{$(0.6, 0.6)$} & \textbf{$(0.8, 0.8)$} \\
        \midrule
        $e_u$ & $2.7565\times 10^{-3}$ & $1.5546\times 10^{-3}$ & $2.4377\times 10^{-3}$ & $1.8518\times 10^{-2}$ \\
        $e_v$ & $4.1229\times 10^{-2}$ & $2.3122\times 10^{-2}$ & $4.0241\times 10^{-2}$ & $3.6876\times 10^{-1}$ \\
        \bottomrule
    \end{tabular}
    \label{table2}
\end{table}

{The results show that the reconstruction errors generally increase with the spatial fractional order, although a non-monotone behavior (an initial decrease followed by an increase) is observed at low orders. This may be attributed to the choice of regularization parameters and to discretization effects. Although the stability index obtained in this paper does not exhibit a clear dependence on the fractional order, the overall trend of the errors is still qualitatively consistent with the theoretical results in \cite{SONG2024177}}: larger fractional orders amplify high-frequency components more significantly and thus make the problem more severely ill-posed. This phenomenon is particularly pronounced for the velocity component. In the following, we analyze the influence of the initial time \(t_0\) on the reconstruction results to further validate our theoretical analysis.
\begin{example}
    According to the conditional stability analysis, the exponent $\theta$ of the stability is closely related to the time $t_0$. As $t_0$ tends $0$, the value of $\theta$ decreases gradually, indicating that the backward problem becomes more severely ill-posed. To verify our theoretical findings, we test the inversion accuracy of our proposed algorithm at different initial times $t_0$.

    In this numerical experiment, the initial displacement is set as
    $$
    u_0(x,y) = \sin(2\pi x)\sin(2\pi y),
    $$
    which is a relatively high-frequency function. Compared with the previous two examples, this problem is more challenging. For simplicity, we set $a=0$ and $\delta=\beta=0.5$. The noise level is fixed at $\sigma=0.01$, and all other settings remain consistent with those in Example~\ref{exp1}.

    To illustrate the exact solution, we present the solution at $t_0=0.2$ in Figure~\ref{fig:true_u3}. We perform inversions for $t_0=0.1, 0.2, 0.3$ and $0.4$, respectively. The reconstructed results are shown in Figure~\ref{fig:recovered_u3}, and the relative errors $e_u$ and $e_v$ are summarized in Table~\ref{table3}. To illustrate the dependence of the relative reconstruction errors on the initial time $t_0$, we present a line plot in Figure~\ref{fig:error_vs_t0}.
\end{example}

\begin{figure}[ht]
    \centering
    \includegraphics[width=0.5\linewidth]{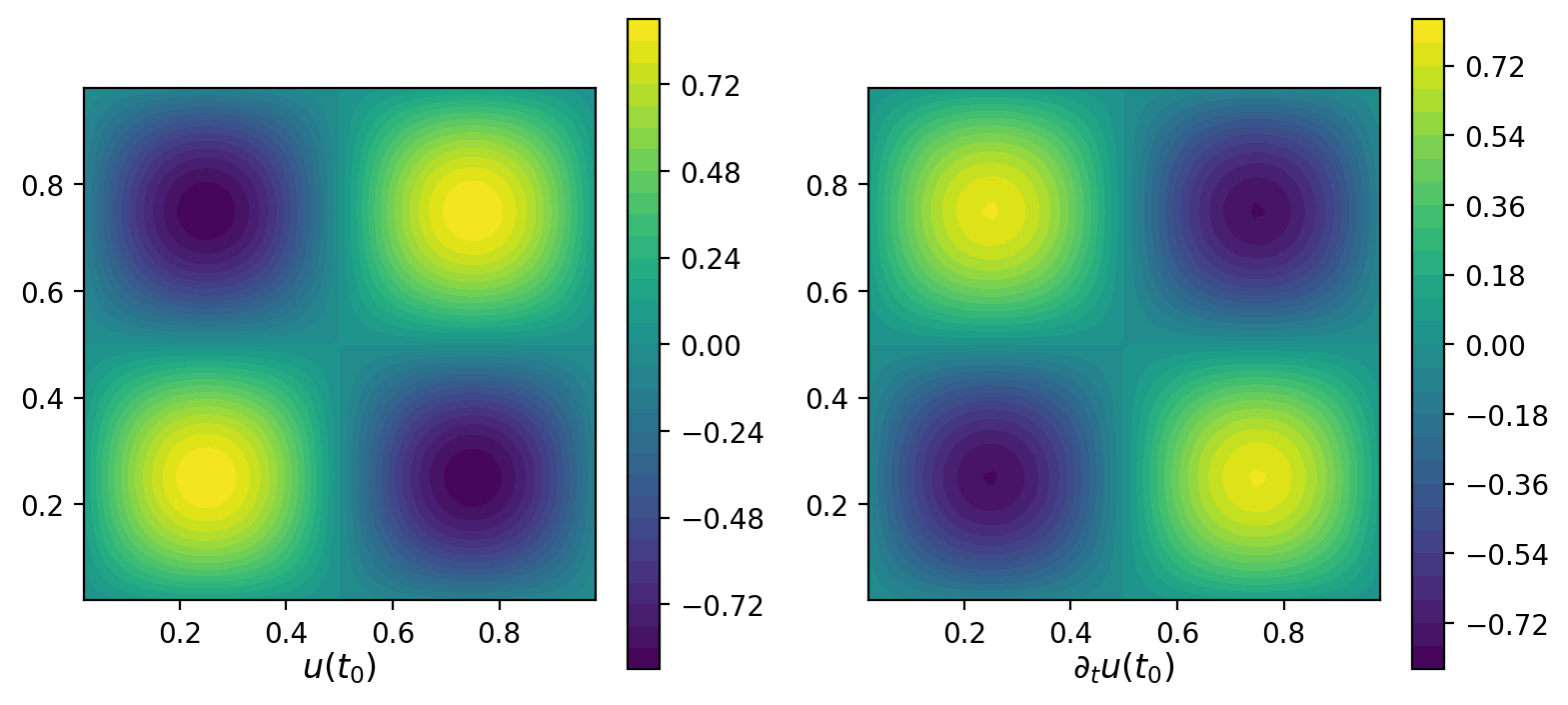}
    \caption{True $u(t_0)$ and $\partial_t u(t_0)$ when $t_0=0.2$}
    \label{fig:true_u3}
\end{figure}

\begin{figure}[ht]
    \centering
    \includegraphics[width=1\linewidth]{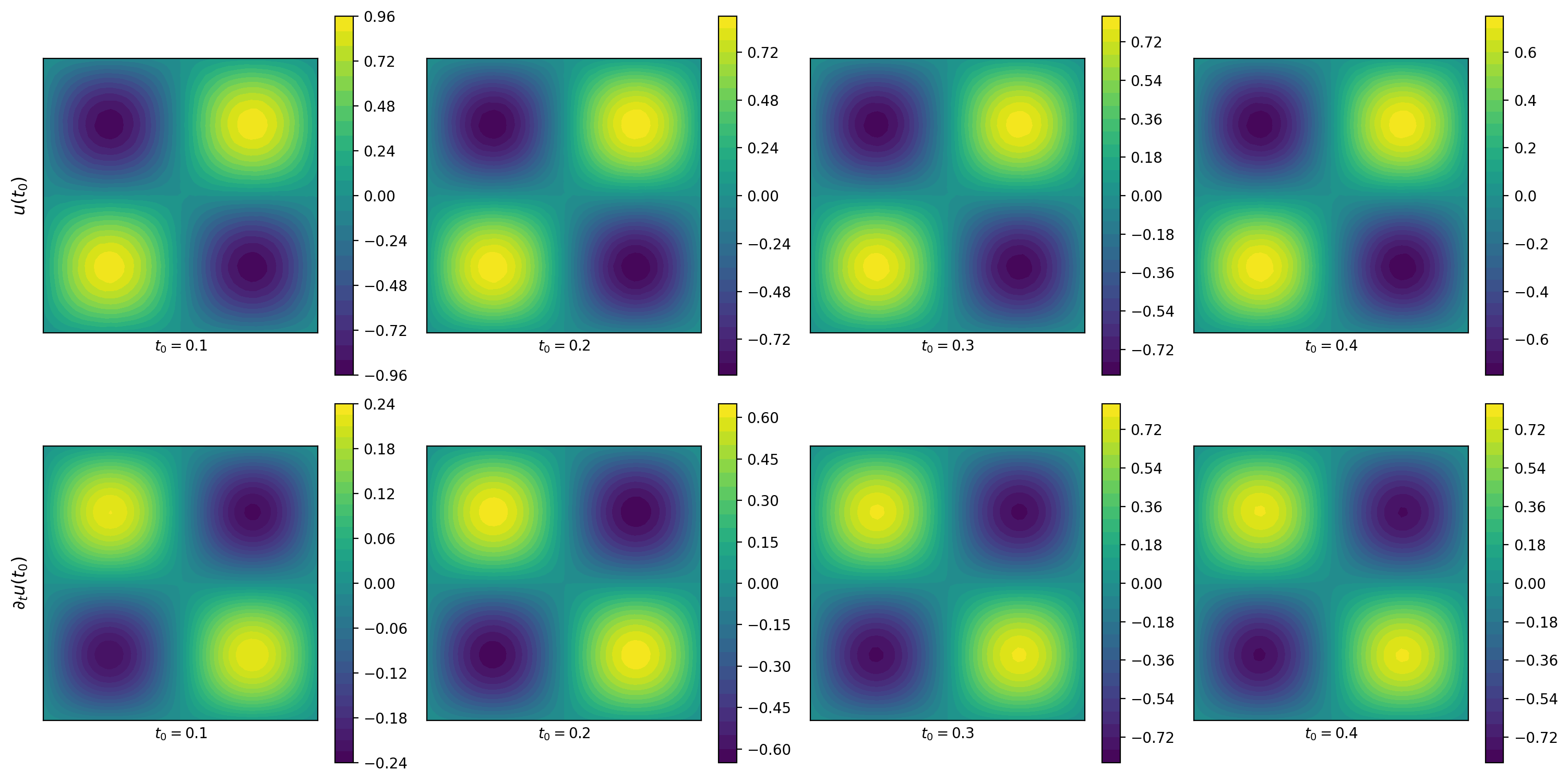}
    \caption{The reconstructions of $u(t_0)$ and $\partial_t u(t_0)$ with $t_0=0.1,0.2,0.3,0.4$}
    \label{fig:recovered_u3}
\end{figure}

\begin{table}[ht]
    \centering
    \caption{Relative errors $e_u$ and $e_v$ at different initial times $t_0$}
    \begin{tabular}{lcccc}
        \toprule
        & \textbf{$t_0=0.1$} & \textbf{$t_0=0.2$} & \textbf{$t_0=0.3$} & \textbf{$t_0=0.4$} \\
        \midrule
        $e_u$ & $3.5845\times 10^{-2}$ & $1.2523\times 10^{-2}$ & $6.5248\times 10^{-3}$ & $6.6968\times 10^{-3}$ \\
        $e_v$ & $6.1160\times 10^{-1}$ & $1.9312\times 10^{-1}$ & $3.8288\times 10^{-2}$ & $9.6682\times 10^{-3}$ \\
        \bottomrule
    \end{tabular}
    \label{table3}
\end{table}

\begin{figure}[ht]
    \centering
    \includegraphics[width=0.5\linewidth]{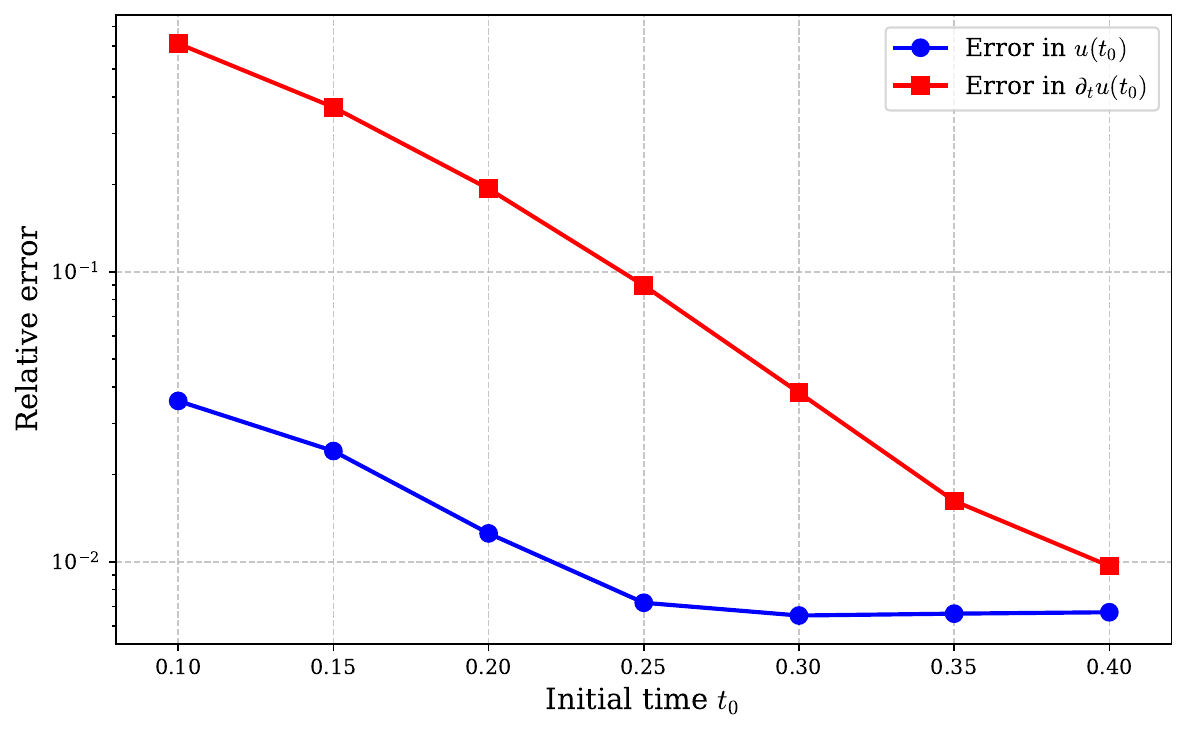}
    \caption{Relative errors $e_u$ and $e_v$ as functions of $t_0$}
    \label{fig:error_vs_t0}
\end{figure}
The results demonstrate that as $t_0$ approaches 0, the reconstruction errors increase significantly, indicating that the problem becomes more severely ill-posed. This is in perfect agreement with our theoretical findings. It can be observed that when $t_0=0.1$, the reconstruction error for the velocity component is extremely large, leading to a significant deviation from the exact solution.

This behavior is analogous to that of parabolic equations, where it is notoriously difficult to reconstruct solutions near the initial time $t=0$ from data at the final time $T$. In contrast, the displacement component can usually be reconstructed with relatively high accuracy. This observation is consistent with the previous two examples, indicating that velocity reconstruction is inherently more challenging. However, a rigorous theoretical explanation for why velocity reconstruction is more difficult than displacement reconstruction is beyond the scope of our current work and will be investigated in our future research.

{
\begin{example}
\label{exp:alpha_sigma2}
We test the parameter rule in Theorem~\ref{thm:main} with $\delta=0.6$, $\beta=0.8$, $a=0$, $T=1$, and $t_0=0.8$, using the initial displacement of Example~\ref{exp1} and $\partial_tu(\cdot,0)=0$.
\end{example}

We use $50\times50$ interior nodes with $h=1/51$ and exact modal propagation. For the terminal data $d=(d_u,d_v)$, draw independent uniform nodal fields $\eta_u,\eta_v\in[-1,1]$, set $r=(\eta_ud_u,\eta_vd_v)$, and define
\[
\begin{aligned}
\|z\|_{X_{0,h}}^2
&=h^2\sum_{i,j=1}^{50}\left(|z_u(x_i,y_j)|^2+|z_v(x_i,y_j)|^2\right),\\
y^{\sigma_k}&=d+\rho_k r,\qquad
\rho_k=\frac{\sigma_k}{\|r\|_{X_{0,h}}}.
\end{aligned}
\]
Thus $\sigma_k=\|y^{\sigma_k}-d\|_{X_{0,h}}$ is the absolute data-norm bound corresponding to the simulated amplitude $\rho_k$. Sine orthogonality identifies this norm with the continuous $X_0$ norm of the interpolants.

We choose $\alpha_k=\sigma_k^2$ and minimize the $L^2$-penalized functional over the discrete counterpart of $\mathcal K_{M,\epsilon}(t_0)$, with
\[
\epsilon=0.2,\qquad
M=1.1\|u^\dagger\|_{H^1(0,T;D(( -\Delta)^{\beta+\epsilon}))}.
\]
The full-history constraint is enforced by a scalar Lagrange multiplier solve. The fixed-$\alpha$ comparison uses the same admissible set with $\alpha=10^{-5}$.

In addition to $e_u$ and $e_v$ in \eqref{eq:numerical_relative_errors}, define the joint absolute error
\begin{equation}\label{eq:numerical_joint_error}
\begin{aligned}
E=\left(
\|u^{\alpha,\sigma}(t_0)-u^\dagger(t_0)\|_{L^2(\Omega)}^2
+\|\partial_tu^{\alpha,\sigma}(t_0)-\partial_tu^\dagger(t_0)\|_{L^2(\Omega)}^2
\right)^{1/2}.
\end{aligned}
\end{equation}
We report means over $20$ trials with seeds $42,\ldots,61$, retaining one normalized noise direction across the noise levels within each trial. Figure~\ref{fig:alpha_sigma2_reconstructions} shows the reconstructions for seed $42$, and Figure~\ref{fig:alpha_sigma2_errors} report the errors. The mean $E$ decreases from $1.041\times10^{-2}$ to $8.861\times10^{-7}$ as $\sigma$ decreases from $10^{-1}$ to $10^{-6}$.

\begin{figure}[ht]
    \centering
    
    \includegraphics[width=\linewidth]{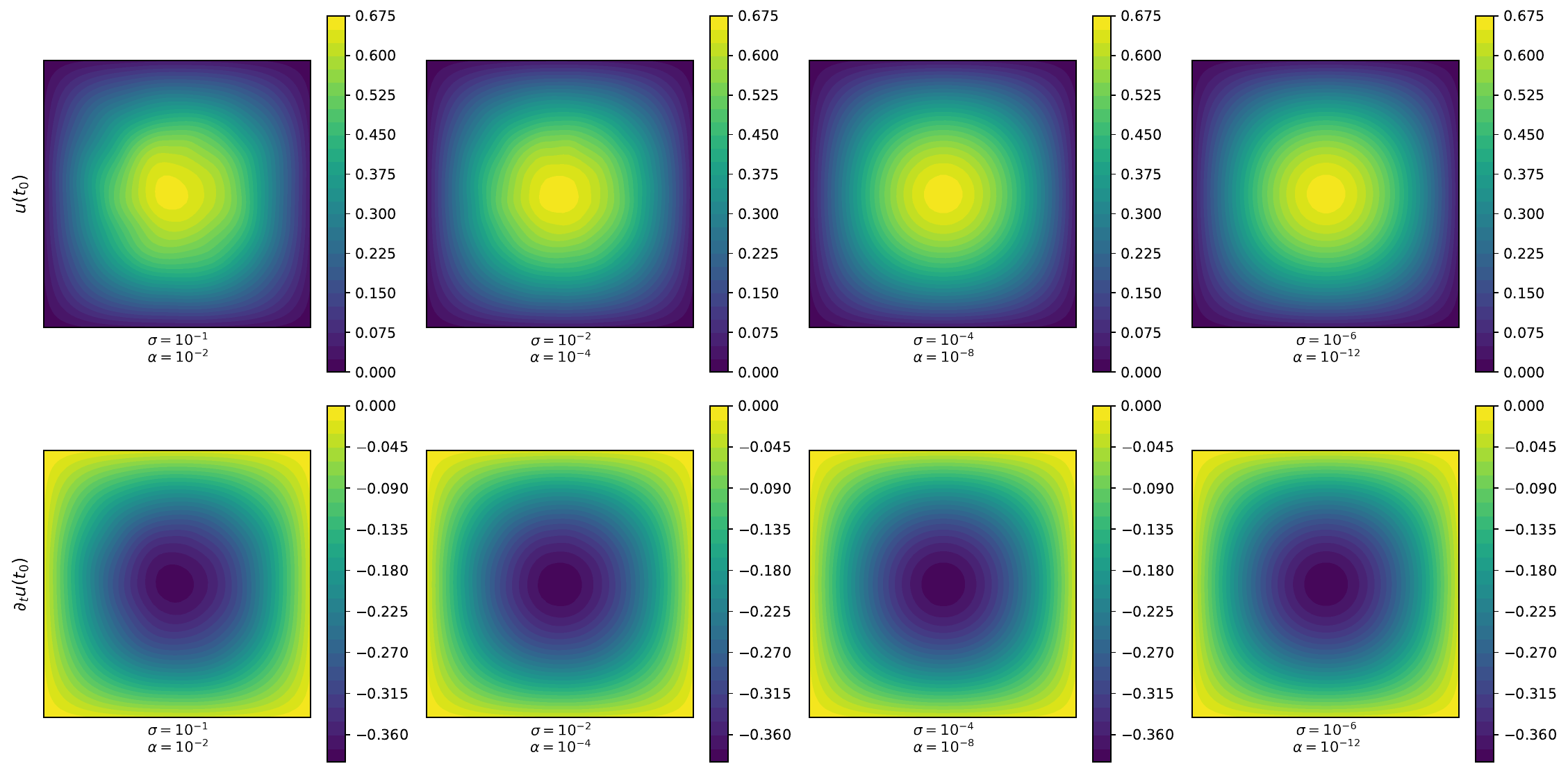}
    \caption{Reconstructions of $u(t_0)$ and $\partial_tu(t_0)$ for $\delta=0.6$, $\beta=0.8$, and $\alpha=\sigma^2$. Each row uses a common color scale.}
    \label{fig:alpha_sigma2_reconstructions}
\end{figure}

\begin{figure}[ht]
    \centering
    
    \includegraphics[width=\linewidth]{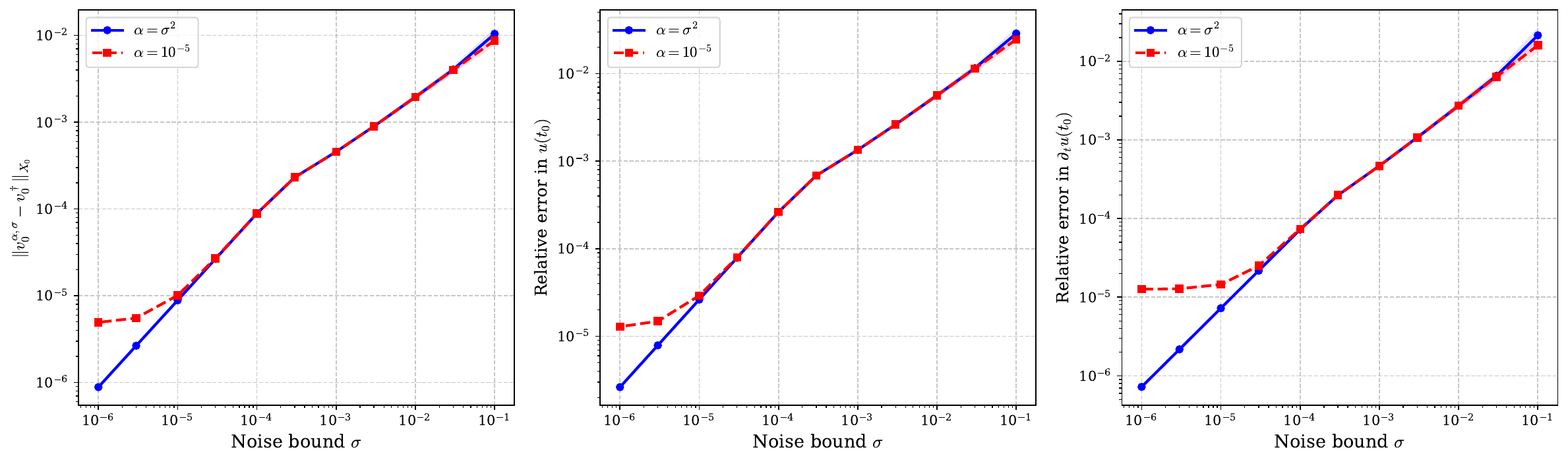}
    \caption{Mean reconstruction errors over $20$ trials for $\alpha=\sigma^2$ and $\alpha=10^{-5}$ in the unequal-order setting. Shaded regions indicate one sample standard deviation.}
    \label{fig:alpha_sigma2_errors}
\end{figure}
}

\section{Concluding remarks}\label{sec7}
In this work, we established a Carleman estimate for a strongly damped wave equation
involving fractional Laplacians \((-\Delta)^\beta\) and \((-\Delta)^\delta\) (\(\delta\le\beta\)) which is motivated by the photoacoustic imaging in attenuating media.
Using the time‑exponential weight \(\varphi(t)=e^{\lambda t}\) and an integral inequality, we proved a conditional stability for the backward problem, with the stability exponent quantifying the ill‑posedness.  Based on this stability estimate, we
constructed a Tikhonov regularization functional together with the corresponding adjoint system, and derived a convergence rate for the proposed algorithm. Numerical experiments are presented, confirming the effectiveness of the theoretical results.

\bibliographystyle{plain}
\bibliography{refs}
\end{document}